\documentclass[reqno,a4paper]{amsart}

\usepackage[latin1]{inputenc}
\usepackage[english]{babel}
\usepackage{eucal,amsfonts,amssymb,amsmath,amsthm,epsfig,mathrsfs}
\usepackage{cancel,soul}
\usepackage{color}
\usepackage{amscd,amsxtra}
\usepackage{enumerate}
\usepackage{latexsym}
\usepackage{bm}

\allowdisplaybreaks

\newcounter{ipotesi}

\Alph{ipotesi}
 \makeatletter \@addtoreset{equation}{section}

\makeatother \makeatletter

\newtheorem{thm}{Theorem}[section]
\newtheorem{hyp}[thm]{Hypotheses}{\rm}
{\rm}
\newtheorem{lemm}[thm]{Lemma}
\newtheorem{cor}[thm]{Corollary}

\newtheorem{prop}[thm]{Proposition}
\newtheorem{defi}[thm]{Definition}
\newtheorem{rmk}[thm]{Remark}{\rm}
\newtheorem{example}[thm]{Example}

\newcounter{parentenv}

\newcommand{\R}{{\mathbb R}}

\newcommand{\E}{{\mathbb E}}
\newcommand{\N}{{\mathbb N}}

\newcommand{\K}{{\mathcal{K}}}

\newcommand{\X}{{\mathcal{X}}}

\newcommand{\J}{{\mathcal{D}}}

\newcommand{\eps}{\varepsilon}
\newcommand{\ra}{\rightarrow}

\newcommand{\ol}[1]{\overline{#1}}

\newcommand{\Tr}{{\operatorname{Tr}}}

\newcommand{\Dom}{{\operatorname{Dom}}}

\newcommand{\Id}{{\operatorname{Id}}}

\newcommand{\set}[1]{{\left\{#1\right\}}}
\newcommand{\pa}[1]{{\left(#1\right)}}
\newcommand{\sq}[1]{{\left[#1\right]}}
\newcommand{\gen}[1]{{\left\langle #1\right\rangle}}
\newcommand{\abs}[1]{{\left|#1\right|}}
\newcommand{\norm}[1]{{\left\|#1\right\|}}
\newcommand{\scal}[2]{{\left\langle #1,#2\right\rangle}}

\newcommand{\eqsys}[1]{{\left\{\begin{array}{ll}#1\end{array}\right.}}
\newcommand{\tc}{\, \middle |\,}

\begin{document}

\frenchspacing

\title[Logarithmic Harnack inequalities for transition semigroups in Hilbert spaces]{Logarithmic Harnack inequalities for transition semigroups in Hilbert spaces}
\author[L. Angiuli, D. A. Bignamini and S. Ferrari ]{Luciana Angiuli, Davide A. Bignamini$^*$, Simone Ferrari}\thanks{$^*$Corresponding author}
\address{L.A. \& S.F.:  Dipartimento di Matematica e Fisica ``Ennio De Giorgi'', Universit\`a del Salento, Via per Arnesano, I-73100 LECCE, Italy}
\address{D.A.B.: Dipartimento di Scienze Matematiche, Fisiche e Informatiche, Plesso di Mate\-matica, Universit\`a degli Studi di Parma, Parco Area delle Scienze 53/A, I-43124 PARMA, Italy}
\email{luciana.angiuli@unisalento.it}
\email{davideaugusto.bignamini@unimore.it}
\email{simone.ferrari@unisalento.it}

\date{\today}

\begin{abstract} We consider the stochastic differential equation
\begin{align*}
\eqsys{dX(t)=[AX(t)+F(X(t))]dt+C^{1/2}dW(t), & t>0;\\
X(0)=x \in \X;}
\end{align*}
where $\X$ is a Hilbert space, $\{W(t)\}_{t\geq 0}$ is a $\X$-valued cylindrical Wiener process, $A, C$ are suitable operators on $\X$ and $F:\Dom(F)\subseteq\X\to \X$ is a smooth enough function. We establish a logarithmic Harnack inequality for the transition semigroup $\{P(t)\}_{t\geq 0}$ associated with the stochastic problem above, under less restrictive conditions than those considered in the literature. Some applications to these inequalities are also shown.
\end{abstract}

\maketitle

\section{Introduction}
The first formulation of the Harnack inequality dates back to $1887$ and can be found in his seminal paper \cite{HAR87},  and concerns positive harmonic functions.
After some partial extensions, the most important contribution is due to J. Moser \cite{MOS61}
which proved the Harnack inequality for positive (weak) solutions of uniformly elliptic linear equations with bounded coefficients in
variational form. Moser also stresses the usefulness of such kind of estimates to deduce regularity results, such as the local h\"olderianity of the solutions. The further passage towards non-linear elliptic equations was made first by J. Serrin \cite{SER55} and then by N.S. Trudinger \cite{TRU} a few years later, and is based on Moser's approach.

The first parabolic version of the Harnack inequality is proved
separately from J. Hadamard \cite{HAD} and B. Pini \cite{PIN} for positive solutions of the heat equation. Many years later this kind of estimates have been extended to positive solutions of more general linear parabolic equations by Moser himself \cite{MOS67}. Hence the extension
to almost linear parabolic equations was due to D. G. Aronson and J. Serrin \cite{AROSERR} and N.S. Trudinger \cite{TRU}. Differently from the elliptic case, however, the
case of operators with non-linear coefficients turned out to be more difficult and remained unresolved for a long time. In this direction we refer to \cite{DiB1,DiB2} where an \emph{intrinsic} Harnack type inequality was proved for solutions of a large class of nonlinear equations and for operators with nonlinear coefficients.
The techniques used in these latter results were inspired by the method of
E. De Giorgi and J. Nash (see \cite{DEGIO57,NASH57}) to show boundedness and regularity for certain classes of
functions (the so-called De Giorgi classes), which contain in particular
the solutions of some elliptic equations.

We refer to \cite{KAS07} and the reference therein for a more in-depth analysis of the Harnack inequality.
In all the quoted results, the formulation of the Harnack inequality allows to compare the values of a positive solution of some elliptic or parabolic differential equation, at two different points. All these Harnack inequalities are dimension-dependent and thus they cannot pass to infinite dimension.
A possibility to get the Harnack-type inequality in an infinite dimensional setting consists in replacing the classical formulation to
the dimension-free logarithmic Harnack Inequality (LHI) firtst introduced by F.-Y. Wang in \cite{WANG97} for the study of diffusion semigroups on a Riemannian manifold $M$. It reads as
\begin{align}\label{HAR_intr}
(P(t)f)^\alpha(x)\leq (P(t)f^\alpha)(x)e^{c(t)\rho(x,y)},\qquad t>0,\ x,y\in M
\end{align}
which holds true for any positive and Borel bounded function $f$, any $\alpha>1$ and some continuous function $c(t)$. Here $\rho$ is a Riemannian metric on $M$. Also in the infinite dimensional setting, this kind of inequality has been used to obtain a lot of results, like some regularizing effects of the semigroup (see, for example, \cite[Proposition 4.1]{DaPROWA09}, \cite[Corollary 1.2]{RO-WA03} and \cite[Corollary 7.3.14]{WANG13}) as well as some hyperboundedness properties for the semigroup $\{P(t)\}_{t\geq 0}$ (see, for example, \cite{RO-WA03} and \cite{WANG04}). We refer to \cite{WANG13} and the reference therein for a discussion of this inequality and its consequences.

The aim of this paper consists in proving a (LHI) like \eqref{HAR_intr}, for transition semigroups associated to some stochastic partial differential equations in infinite dimensional separable Hilbert spaces under more less restrictive conditions than those considered in the literature.

We consider $(\Omega,\mathcal{F},\{\mathcal{F}_t\}_{t\geq0},\mathbb{P})$ a normal filtered probability space, $\X$ a separable Hilbert space with inner product $\gen{\cdot,\cdot}$ and associated norm $\norm{\cdot}$ and the stochastic partial differential equation
\begin{gather}\label{eqFint}
\eqsys{
dX(t)=[AX(t)+F(X(t))]dt+ C^{1/2}dW(t), & t> 0;\\
X(0)=x\in \X,
}
\end{gather}
where $\{W(t)\}_{t\geq 0}$ is a $\X$-valued cylindrical Wiener process, $C\in\mathcal{L}(\X)$ is a positive operator, $F:\Dom(F)\subseteq\X\ra\X$ is a regular enough function, and $A:\Dom(A)\subseteq \X\ra\X$ is a possibly unbounded operator. Throughout the paper we will assume some hypotheses on $A$, $F$ and $C$ to guarantee the existence of a mild solution for \eqref{eqFint} (see Hypotheses \ref{EU2} and \ref{HA}). Stochastic partial differential equations like \eqref{eqFint} are widely studied in the literature (see \cite{DA-ZA3,WANG13} and the references therein) as well as the validity of some (LHI) in Hilbert spaces for the associated semigroup
\begin{align}\label{semi_intr}
P(t)\varphi(x):=\mathbb{E}[\varphi(X(t,x))],\qquad x\in\X,\ t\geq 0,\ \varphi\in B_b(\X);
\end{align}
(see, for example, \cite{DaPROWA09,RO-WA1}). Here $\mathbb{E}[\cdot]$ denotes the expectation with respect to $\mathbb{P}$ and $B_b(\X)$ is the space of bounded and Borel measurable function on $\X$. Estimates like \eqref{HAR_intr} for the transition semigroup \eqref{semi_intr} can be found for instance in \cite{DaPROWA09,ES09,HUZA19,LVHU21,RO-WA1,WANGYUAN11}. In all the quoted papers two different sets of assumptions for $A, F$ and $C$ are made in order to get inequalities like \eqref{HAR_intr}.
 Concerning the operator $C$, it is required some sort of invertibility. In \cite{DaPROWA09,WANGYUAN11} it is required that $C$ itself admits continuous and bounded linear inverse. In \cite{HUZA19,LVHU21} it is assumed that $CC^*$ is invertible, while in \cite{ES09} the authors restrict themselves to consider as $C$ the identity operator.
 On the other hand, for what concerns the function $F$, it is usually required that it is Lipschitz continuous and that it satisfies the following dissipativity type condition:
\begin{align}\label{RO-WA_hyp}
\langle F(x)-F(y),C^{-1}(x-y) \rangle\leq \zeta\|C^{-1/2}(x-y)\|,\qquad x-y\in C(\X).
\end{align}
for some $\zeta \in \R$. This latter condition can be found in \cite{RO-WA1} where $C$ actually can depend also on $x$.

The main results of this paper are stated in Theorems \ref{lip_harnak} and \ref{Harnak_diss} where a (LHI) similar to \eqref{HAR_intr} is obtained for the transition semigroup \eqref{semi_intr} without any hypotheses of invertibility on $C$ and assuming only one between the Lipschitz and the dissipativity condition appearing in \eqref{RO-WA_hyp} (see Section \ref{liphar} for the Lipschitz continuous case and Section \ref{nonome} for the dissipative case).  To be more precise we will prove a (LHI) type inequality along the direction of the square root of the diffusion operator $C$, namely
\begin{align}\label{HAR_alongHC_intr}
|P(t)\varphi(x+h)|^p\leq P(t)|\varphi(x)|^pe^{c(t)\|C^{-1/2}h\|^2},\qquad t>0,\ x \in \X,\ h \in C^{1/2}(\X);
\end{align}
for any  bounded and Borel measurable function $\varphi:\X\ra\R$, any $p>1$ and some continuous function $c:(0,+\infty)\ra\R$. We point out that if $C$ has a continuous inverse than \eqref{HAR_alongHC_intr} is equivalent to \eqref{HAR_intr}.

The key tool we use to prove the (LHI) \eqref{lip_harnak_est} and \eqref{dae} in both cases is an approximation method. In the Lipschitz continuous case the approximants which allow us to get our estimate are suitable finite dimensional semigroups which satisfy suitable gradient estimates. On the other hand, the dissipative case is solved by using a double approximation procedure which consider a finite dimensional approximation of the Yosida approximants.

For our results, we have collected some standard consequences in Section \ref{sect_examples}.
For instance, the (LHI) can be used to prove a strong Feller-type property for $\{P(t)\}_{t\geq 0}$ and, assuming further the existence of an invariant measure $\mu$ for $\{P(t)\}_{t\geq 0}$, the occurrence of an entropy-cost inequality (see Corollary \ref{cor-cons} (iii)). In this last case, another classical consequence of the (LHI) is a hypercontractivity type estimate for the semigroup $\{P(t)\}_{t\geq 0}$ in $L^p(\X,\mu)$. Such estimate relies on the H\"older inequality and some integrability conditions with respect to $\mu$ of some exponential functions (see Corollary \ref{cor-hyp}).
Finally, a considerable set of examples of operators $A, C$ and functions $F$ to which our results can be applied, are collected at the end of Section \ref{sect_examples}.

\section{Notations and preliminaries}
In this section we fix the notations and we recall some basic results we will use throughout the paper.

Let $H_1$ and $H_2$ be two Hilbert spaces with inner products $\gen{\cdot,\cdot}_{H_1}$ and $\gen{\cdot,\cdot}_{H_2}$ and associated norms $\norm{\cdot}_{H_1}$ and $\norm{\cdot}_{H_2}$, respectively. We denote by $B_b(H_1;H_2)$ (resp. $C_b(H_1;H_2)$) the set of functions $f:H_1 \to H_2$ which are bounded and Borel measurable (resp. continuous). When $H_2=\R$ we simply write $B_b(H_1)$ (resp. $C_b(H_1)$). For any $k \in \N\cup\{\infty\}$, $C^k_b(H_1;H_2)$ consists of continuous functions $f:H_1\to H_2$ which are $k$-times Fr\'echet differentiable with bounded and continuous derivatives up to order $k$. If $H_2=\R$ we simply write $C_b^k(H_1)$ and for $k=1$ then, for any $x \in H_1$, we denote by $\J f(x)$ the unique $k\in H_1$ such that
\[\lim_{\norm{h}_{H_1}\ra 0}\frac{|f(x+h)-f(x)-\gen{h,k}_{H_1}|}{\norm{h}_{H_1}}=0.\]
For any $k\in\N\cup\set{\infty}$, we denote by $\mathcal{F}C_b^{k,n}(H_1)$, the space of cylindrical $C^k_b$ functions depending on $n$ variables, i.e., the set of functions $f:H_1\to \R$ such that $f(x)=\varphi(\langle x, h_1\rangle_{H_1},\ldots, \langle x, h_n\rangle_{H_1})$, $x \in H_1$, for some  $\varphi \in C_b^k(\R^n)$, $h_1,\ldots, h_n\in H_1$ and $n\in\N$.

Let $G: \Dom(G)\subseteq H_1\ra H_1$ be a linear operator and let $H_2\subseteq H_1$. We call \emph{part of $G$ in $H_2$} the operator $G_{H_2}:\Dom(G_{H_2})\subseteq H_2\ra H_2$ defined as
\begin{align*}
\Dom(G_{H_2})&:=\{x\in \Dom(G)\cap H_2\,|\, Gx\in H_2\};\\
G_{H_2}x&:=Gx,\qquad x\in \Dom(G_{H_2}).
\end{align*}
By $\mathcal{L}(H_1)$ we denote the set of all bounded linear operators from $H_1$ into itself and by $\Id_{H_1}\in\mathcal{L}(H_1)$ the identity operator on $H_1$. We say that $B\in\mathcal{L}(H_1)$ is \emph{non-negative} (resp. \emph{positive}) if for every $x\in H_1\setminus\set{0}$
\[
\gen{Bx,x}_{H_1}\geq 0\ (>0).
\]
On the other hand, $B \in \mathcal{L}(H_1)$ is a \emph{non-positive} (resp. \emph{negative}) operator, if $-B$ is non-negative (resp. positive). We remind the reader that such operators are self-adjoint (see \cite[Section IV.4]{RE-SI1}).
We recall that if a semigroup $\{P(t)\}_{t\geq 0}\subseteq\mathcal{L}(H_1)$ is strongly continuous, then there exists $M_0\geq 1$ and $\eta_0\in\R$ such that
\begin{equation*}
\norm{P(t)}_{\mathcal{L}(H_1)}\leq M_0e^{\eta_0t}, \qquad\;\, t \geq 0.
\end{equation*}
For the general theory concerning linear operators and semigroups we refer to \cite{DUN-SCH1,DUN-SCH2,EN-NA1}.

Now, let $(\Omega,\mathcal{F},\{\mathcal{F}_t\}_{t\geq 0},\mathbb{P})$ be a normal filtered probability space, $\K$ be a separable Banach space and by $\mathcal{B}(\K)$ we denote the family of the Borel subsets of $K$. For any random variable $\xi:(\Omega,\mathcal{F},\mathbb{P})\ra (\K,\mathcal{B}(\K))$, the \emph{law of $\xi$} on $(\K, \mathcal{B}(\K))$, denoted by $\mathscr{L}(\xi)$ and the \emph{expectation of $\xi$} with respect to $\mathbb{P}$ denoted by $\mathbb{E}[\xi]$ are defined by the formulas
\[
\mathscr{L}(\xi):=\mathbb{P}\circ\xi^{-1}
\]
and
\[
\mathbb{E}[\xi]:=\int_\Omega \xi(w)\  \mathbb{P}(d\omega)=\int_\K x\  \mathscr{L}(\xi)(dx).
\]
We need to recall the definition of some Banach spaces often used in the literature (see, for example, \cite[Section 6.2]{CER1}).
\begin{defi}\label{defi_space}
Let $I \subseteq [0,+\infty)$ be an interval and let $p\geq 1$. We denote by $C_p(I;\K)$ the space of $\K$-valued processes $\{Y(t)\}_{t\in I}$ such that the function $(Y(\cdot))(\omega):I\ra \K$ is bounded and continuous for $\mathbb{P}$-a.e. $\omega \in \Omega$. We endow the space $C_p(I;\K)$ with the norm
\[
\|\{Y(t)\}_{t\in I}\|_{C_p(I;\K)}^p:=\E\sq{\sup_{t\in I}\norm{Y(t)}_\K^p}.
\]
We denote by $\K^p(I)$ the space of progressively measurable\footnote{We say that a $\K$-valued process $\psi$ is \emph{progressively measurable} if, as a transformation from $\Omega \times [0,t]$ equipped with the $\sigma$-field $\mathcal{F}_t \times \mathcal{B}([0,t])$ into $(\K, \mathcal{B}(\K))$ is measurable for any $[0,t]\subset I$.} $\K$-valued process $\{Y(t)\}_{t\in I}$ such that the quantity
\[
\|\{Y(t)\}_{t\in I}\|_{\K^p(I)}^p:=\sup_{t\in I}\E\Big[\norm{Y(t)}_\K^p\Big],
\]
is finite. We endow the space $\K^p(I)$ with the norm $\norm{\cdot}_{\K^p(I)}$,
\end{defi}


A map $f:\Dom(f)\subseteq \mathcal{K}\ra\mathcal{K}$ is said to be \emph{dissipative} if for any $\alpha>0$ and $x,y\in \Dom(f)$, it holds
\begin{equation}\label{disban}
\norm{x-y-\alpha(f(x)-f(y))}_{\mathcal{K}}\geq \norm{x-y}_{\mathcal{K}}
\end{equation}
In particular if $\mathcal{K}$ is a Hilbert space \eqref{disban} is equivalent to
\begin{equation*}
\scal{f(x)-f(y)}{x-y}_{\mathcal{K}}\leq 0.
\end{equation*}
We say that $f$ is \emph{$m$-dissipative} if it is dissipative and the range of ${\rm Id}_{\K}-f$ is the whole space $\K$.

\begin{rmk}
If $f=A$ is linear, then \eqref{disban} becomes $\norm{(\lambda\Id-A)x}_\K\geq \lambda \norm{x}_\K$ for every $x\in\Dom(A)$ and $\lambda>0$. In addition, if $\K$ is a Hilbert space with inner product $\gen{\cdot,\cdot}_{\K}$, then \eqref{disban} is equivalent to say that $\langle A h,h\rangle_{\K} \le 0$ for any $h \in \K$.
\end{rmk}

Now we introduce a Hilbert space that plays a key role int the proof of the (LHI) (see Sections \ref{liphar} and \ref{nonome}). In what follows, $\X$ will be an Hilbert space with inner product $\gen{\cdot,\cdot}$ and associated norm $\norm{\cdot}$, and $C\in \mathcal{L}(\X)$ will be a positive operator.
We set
\begin{equation*}
H_C:=C^{1/2}(\X), \qquad [h_1,h_2]_C:=\langle C^{-1/2}h_1,C^{-1/2}h_2\rangle,\qquad h_1,h_2\in C^{1/2}(\X).
\end{equation*}
It can be proved that $(H_C,[\cdot,\cdot]_C)$ is a Hilbert space continuously embedded in $\X$ (see \cite{BI-FE1}). We set $\norm{\cdot}_C$ the norm on $H_C$ associated to the inner product $[\cdot,\cdot]_C$.

We give the definition of $H_C$-Fr\'echet differentiability.

\begin{defi}
A function $\Phi: \X\rightarrow \R$ is $H_C$-Fr\'echet differentiable at $x\in\X$, if there exists $h_x\in H_C$ such that
\begin{align*}
\lim_{\norm{h}_C\ra 0}\frac{\left|\Phi(x+h)-\Phi(x)-[h_x,h]_C\right|}{\norm{h}_C}=0.
\end{align*}
Clearly, when $h_x$ exists, it is unique and we set $\J_C\Phi(x):=h_x$.
\end{defi}

\begin{rmk}\label{sun}
{\rm Observe that if $\Phi:\X\ra\R$ is a Fr\'echet differentiable function, then it is $H_C$-differentiable too. Furthermore, for every $x\in\X$ it holds
\[\J_C \Phi(x)=C \J\Phi(x),\]
This result can be found in \cite[Proposition 17]{BI-FE1}.}
\end{rmk}

The definition of $H_C$-Gateaux differentiability is the Gateaux counterpart of the previous definition. We stress that for our purposes we need a more technical definition in comparison to the ``standard'' one.

\begin{defi}
We say that a function $\Phi:\X\ra\X$ is $H_C$-Gateaux differentiable if for every $x\in\X$ and $h\in H_C$, there exists $\eps_{x,h}>0$ such that the function $\varphi_{x,h}:(-\eps_{x,h},\eps_{x,h})\ra\X$ defined as
\[\varphi_{x,h}(r):=\Phi(x+rh)-\Phi(x)\]
is $H_C$-valued and there exists $L_x\in \mathcal{L}(H_C)$ such that for every $h\in H_C$
\begin{align}
\lim_{r\ra 0}\norm{\frac{1}{r}\varphi_{x,h}(r)-L_xh}_C=0.\label{lim_dif}
\end{align}
\end{defi}

\noindent See \cite{BI-FE1} for further properties of the space $H_C$.

We end this section by recalling a result of existence and uniqueness of the mild solution of the problem
\begin{equation}\label{SDE}
\left\{\begin{array}{ll}
dX(t)=[AX(t)+F(X(t))]dt+ C^{1/2}dW(t),&  t> 0;\\
X(0)=x\in \X.
\end{array}\right.
\end{equation}
Here $\{W(t)\}_{t\geq 0}$ is a $\X$-valued cylindrical Wiener process (see \cite[Section 4.1.2]{DA-ZA3} for a definition) and the operator $A,F$ and $C$ satisfy the following assumptions.

\begin{hyp}\label{EU2}
Let $\X$ be a Hilbert space with inner product $\gen{\cdot,\cdot}$ and associated norm $\norm{\cdot}$.
 \begin{enumerate}[\rm(i)]
\item There exists $E$ a Borel subset of $\X$ which is a Banach space continuously and densely embedded in $\X$;
\item $A:\Dom(A)\subseteq\X\ra\X$ generates a strongly continuous semigroup $e^{tA}$ on $\X$ and $A_E$ (the part of $A$ in $E$) generates an analytic semigroup $e^{tA_E}$ on $E$. There exists $\zeta_A\in\R$ such that $A-\zeta_A\Id_{\X}$ is dissipative in $\X$ and $A_E-\zeta_A\Id_{E}$ is dissipative in $E$. For any $T>0$ and for $\mathbb{P}$-a.e. $\omega \in \Omega$ the function $W_{A}(\cdot)(\omega):[0,T]\to E$ defined by
\begin{equation}\label{stoconv}
W_{A}(t)(\omega):=\left(\int^t_0e^{(t-s)A} C^{1/2}dW(s)\right)(\omega),\qquad t\in[0,T],
\end{equation}
is continuous and
\begin{equation}\label{trace}
\int^T_0\Tr[e^{sA}Ce^{sA^*}]ds<+\infty.
\end{equation}

\item $E\subseteq \Dom(F)$ and $F(E)\subseteq E$. There exists $\zeta_F\in\R$ such that $F-\zeta_F\Id_{\X}$ is $m$-dissipative in $\X$ and $F_{|_E}-\zeta_F\Id_{E}$ is m-dissipative in $E$. $F_{|_E}:E\ra E$ is locally Lipschitz on $E$, namely $F_{|_E}$ is Lipschitz continuous on bounded subsets of $E$ and, there exist $M>0$ and $m\in\N$ such that
\begin{equation}\label{noname}
\|F_{|_E}(x)\|_E\leq M(1+\norm{x}^m_E),\qquad x\in E.
\end{equation}
\item
$C \in \mathcal{L}(\X)$ is a positive operator.
\end{enumerate}
\end{hyp}

\begin{rmk}
If $E=\X$, by \cite[Theorem 5.11]{DA-ZA3}, a sufficient condition that guarantees both the continuity of the trajectories of $W_A$ and \eqref{trace} is provided by the following: there exists $\eta \in (0,1)$ such that
\begin{equation}\label{condprato}
\int_0^T s^{-\eta}\Tr[e^{sA}Ce^{sA^*}]ds<+\infty.
\end{equation}
If $E \neq \X$, other conditions can be found in \cite[Sections 6.1-8.2]{CER1} and \cite[Sections 5.4-5.5]{DA-ZA3}.
\end{rmk}

For a $x$ belonging to $\X$, a mild solution of \eqref{SDE} is a $\X$-valued adapted stochastic process $\set{X(t,x)}_{t\geq 0}$ satisfying
\begin{equation*}
X(t,x)=e^{tA}x+\int_0^te^{(t-s)A}F(X(s,x))ds+\int_0^te^{(t-s)A}C^{1/2}dW(s),\qquad t\geq 0.
\end{equation*}
The following is a result of existence and uniqueness of a mild solution of \eqref{SDE} whose proof can be found in \cite[Section 3]{BI1}.

\begin{thm}\label{Genmild}
Assume Hypotheses \ref{EU2} hold true.
\begin{enumerate}[\rm (i)]
\item For any $x\in E$, \eqref{SDE} has a unique mild solution $\{X(t,x)\}_{t\geq 0}$ belonging to $C_p([0,T],\X)\cap C_p((0,T],E)$, for any $T>0$ and $p\geq 1$.

\item For any $x\in\X$, \eqref{SDE} has a unique \emph{generalized mild solution} $\{X(t,x)\}_{t\geq 0}$ belonging to $C_p([0,T],\X)$, for any $T>0$ and $p\geq 1$, namely, for any sequence $\{x_n\}_{n\in\N}\subseteq E$ converging to $x$ in $\X$ and for any $T>0$ it holds
\begin{equation*}
\lim_{n\ra+\infty}\sup_{t\in [0,T]}\norm{X(t,x)-X(t,x_n)}=0,
\end{equation*}
where $\{X(t,x_n)\}_{t\geq 0}$ is the unique mild solution of \eqref{SDE} with initial datum $x_n$ (see (i)).
\end{enumerate}
Moreover for any $p\geq 1$, there exist two positive constants $C_p,\kappa_p$ such that
\begin{equation*}
\norm{X(t,x)}_{\mathcal{O}}^p \leq  C_p\left(e^{-\kappa_p t}\norm{x}^p_{\mathcal{O}}+\norm{W_A(t)}_{\mathcal{O}}^p+\int_0^te^{-\kappa_p (t-s)}\norm{F(W_A(s))}^p_{\mathcal{O}}ds\right),
\end{equation*}
for any $t>0$ and $x\in {\mathcal{O}}$, being either ${\mathcal{O}}=E$ or ${\mathcal{O}}=\X$.
Finally, there exists $\eta>0$ such that
\begin{align}
\norm{X(t,x_1)-X(t,x_2)}&\leq e^{-\eta t}\norm{x_1-x_2}, \qquad t>0,\ x_1,x_2 \in \X.\label{lipdete}
\end{align}
\end{thm}

\noindent We point out that all the (in)equalities in the above theorem and in what follows, if not otherwise specified, have to be understood to hold for $\mathbb{P}$-a.e. $\omega\in\Omega$.

\section{Logarithmic Harnack inequality: the Lipschitz continuous case}\label{liphar}

(LHI) are a classical tools to prove regularity results for transition semigroups associated to stochastic differential equations. They have already been proved in several settings and more is known in finite dimension. In this case, (LHI) are satisfied by solutions of the Cauchy problems associated to the equation
\[\frac{d}{dt}u(t,\xi) = \mathcal{A}(t)u(t,\xi),\qquad t>0,\ \xi\in\R^n;\]
where $\mathcal{A}(t)\psi:= {\rm Tr}(Q(t)D^2\psi)+\langle b(t,x), D\psi\rangle$ for any smooth function $\psi$ assuming suitable conditions on $Q$ and $b$ (see, for example, \cite{DENG18,ES09,Gordina11,WANGYUAN11}). In infinite dimension such kind of estimates have been already proved in \cite{DaPROWA09} and in \cite{RO-WA1} for the transition semigroups associated to \eqref{SDE}. In \cite{DaPROWA09}, (LHI) are proved assuming that $C^{-1}\in \mathcal{L}(\X)$ and the dissipativity of $F$ whereas in \cite{RO-WA1} they are proved assuming that $C$ depends on the spatial variable and $F$ is Lipschitz continuous and dissipative along both $\X$ and $H_C$. Here we generalize the results in \cite{RO-WA1} when $C \in \mathcal{L}(X)$ assuming either $F$ Lipschitz continuous or dissipative along $H_C$ and in this last case without requiring that  $C^{-1}\in \mathcal{L}(\X)$.

We start by assuming the following set of hypotheses.
\begin{hyp}\label{HA}
Hypotheses \ref{EU2} are satisfied with $E=\X$. Furthermore we assume that $F:\X\ra\X$ is a Fr\'echet differentiable and Lipschitz continuous function with Lipschitz constant $L_F$, and  there exists $\zeta_\X>0$ such that
\begin{align}\label{diss_X}
\langle (A+\J F(x))h,h\rangle\leq -\zeta_\X\norm{h}^2, \qquad\;\, x,h\in\X.
\end{align}
\end{hyp}
\begin{rmk}{\rm
We point out that the Lipschitz continuity of $F$ together with \eqref{diss_X} yields the existence of $\ol{\zeta} \in \R$ such that $A-\ol{\zeta} Id_{\X}$ and $F-\ol{\zeta} Id_{\X}$ are both dissipative. Moreover, we remark that the following weaker form of \eqref{diss_X} can be assumed in Hypotheses \ref{HA} instead of \eqref{diss_X}
\begin{align*}
\langle A(h_1-h_2)+F(x+h_1)-F(x+h_2),h_1-h_2\rangle\leq -\zeta_\X\norm{h_1-h_2}^2,\qquad h_1,h_2\in\X.
\end{align*}
For the sake of clarity in the proofs of this section we decided to assume \eqref{diss_X}.}
\end{rmk}

Under Hypotheses \ref{HA} the stochastic problem \eqref{SDE} admits a unique mild solution $\{X(t,x)\}_{t\geq 0}$
satisfying \eqref{lipdete} and allows us to define the family of operators $\{P(t)\}_{t\geq 0}$ as
\begin{align}\label{trans}
P(t)\varphi(x):=\E[\varphi(X(t,x))],\qquad x\in\X,\ t\geq 0,\ \varphi\in B_b(\X).
\end{align}
Further regularity properties of the mild solution $\{X(t,x)\}_{t\geq 0}$ of \eqref{SDE} will be needed in the sequel.

\begin{prop}\label{prodiff}
Assume Hypotheses \ref{HA} hold true and let $T>0$. The map $x\mapsto \{X(t,x)\}_{t\in[0,T]}$ from $\X$ into $\X^2([0,T])$ (see Definition \ref{defi_space}) is Gateaux differentiable and for any $x,y\in\X$, its Gateaux derivative is the unique mild solution of
\begin{gather}\label{dermildX}
\eqsys{
dY_x(t)=[A+ \J F(X(t,x))]Y_x(t)dt, & t\in(0,T];\\
Y_x(0)=y,
}
\end{gather}
namely, for every $x,y\in \X$ the process $\{\J_G X(t,x)y\}_{t\in[0,T]}$ satisfies
\begin{equation}\label{DmildY}
\J_G X(t,x)y=e^{tA}y+\int^t_0e^{(t-s)A}\J F(X(s,x))\J_G X(s,x)yds,\qquad t\in[0,T].
\end{equation}
In addition it holds that
\begin{align}\label{stigraX}
\|\J_G X(t,x)y\|\leq e^{-\zeta_\X t}\norm{y},\qquad t\in[0,T],\ x,y \in \X.
\end{align}
\end{prop}

\begin{proof}
The first part of the statement is classical and can be found in \cite[Theorem 9.8]{DA-ZA3} or \cite{MAS08}. To prove estimate \eqref{stigraX} we assume that $\{Y_x(t,y)\}_{t\geq 0}$ is a strict solution of \eqref{dermildX}, otherwise we can proceed as in \cite[Proposition 3.6]{BI1} or \cite[Proposition 6.2.2]{CER1} approximating $\{Y_x(t,y)\}_{t\geq 0}$ by means of a sequence of more regular processes. In this case, scalarly multiplying the stochastic differential equation in \eqref{dermildX} by $Y_x(t,y)$ and using \eqref{diss_X} we obtain
\begin{align*}
\frac{1}{2}\frac{d}{dt}\norm{Y_x(t,y)}^2=\langle [A+\J  F(X(t,x))]Y_x(t,y),Y_x(t,y)\rangle\leq -\zeta_\X\norm{Y_x(t,y)}^2,
\end{align*}
whence we infer \eqref{stigraX}.
\end{proof}


Our proof of (LHI) relies on a gradient estimate for the transition semigroup $\{P(t)\}_{t\geq 0}$ (see \eqref{V1}). To obtain such result we need some additional hypotheses on $A,C$ and $F$.

\begin{hyp}\label{Sinem}
Assume that the part of $A$ in $H_C$, denoted by $A_C$, generates a contractive strongly continuous semigroup $e^{tA_C}$ in $H_C$ and that one of the following two conditions is satisfied:
\begin{enumerate}[\rm (i)]
\item\label{Luz} there exists $\gamma\in(0,1)$ such that
\begin{align}\label{case1}
e^{tA}(\X)\subseteq C^{1/2}(\X),\qquad \|C^{-1/2}e^{tA}\|_{\mathcal{L}(\X)}\leq K t^{-\gamma}
\end{align}
for any $t>0$ and some positive constant $K$ independent of $t$;
\item $F=C^{1/2}G$ for some Fr\'echet differentiable and Lipschitz continuous function $G:\X\ra\X$ with Lipschitz constant $L_G$.
\end{enumerate}
\end{hyp}

\noindent We refer to Section \ref{sect_examples} for some examples of operators $A$ and $C$ verifying Hypotheses \ref{Sinem}.

The following proposition gives us information about the continuity of $h\mapsto \J_GX(t,x)h$ as a linear map in $H_C$.

\begin{prop}\label{stimild1}
Assume Hypotheses \ref{HA} and \ref{Sinem} hold true. For any $t>0$ and $x \in X$, the operator $\J_GX(t,x)$ belongs to $\mathcal{L}(H_C)$. More precisely
\begin{align}\label{sti}
\|\J_GX(t,x)h\|_{C}\leq K_{1}\norm{h}_{C}, \qquad t>0,\ x \in \X,\, h \in H_C,
\end{align}
for some positive constant $K_{1}$ independent of $t,x$ and $h$.
\end{prop}

\begin{proof}
We start by assuming that Hypotheses \ref{Sinem}(i) hold true. For any $t>0$, $x\in \X$ and $h\in H_{C}$, by \eqref{DmildY} and the definition of $\norm{\cdot}_C$,
\begin{align*}
\norm{\J_GX(t,x)h}_{C}\leq \|e^{tA_C}h\|_C+\int^t_0\|C^{-1/2}e^{(t-s)A}\J F(X(s,x))\J_G X(s,x)h\|ds
\end{align*}
By the Lipschitz continuity of $F$, the contractivity of $\{e^{tA_C}\}_{t\geq 0}$ in $H_C$, Hypotheses \ref{Sinem}\eqref{Luz} and \eqref{stigraX} we have
\begin{align*}
\|\J_GX(t,x)h\|_{C}&\leq \norm{h}_{C}+K L_F\|C^{1/2}\|_{\mathcal{L}(\X)}\norm{h}_{C}\int^t_0\frac{e^{-\zeta_\X s}}{(t-s)^{\gamma}}ds\\
&\leq \norm{h}_{C}+K L_F\|C^{1/2}\|_{\mathcal{L}(\X)}\norm{h}_{C}\pa{\int^1_0\frac{1}{s^{\gamma}}ds+\int^t_1e^{-\zeta_\X (t-s)}ds}\\
&\leq \sq{1+K L_F\|C^{1/2}\|_{\mathcal{L}(\X)}\pa{\frac{1}{1-\gamma}+\frac{1}{\zeta_X}}}\norm{h}_{C}
\end{align*}
whence the claim, with $K_1:=1+K L_F\|C^{1/2}\|_{\mathcal{L}(\X)}((1-\gamma)^{-1}+\zeta_{\X}^{-1})$.

Now we assume that Hypotheses \ref{Sinem}(ii) are satisfied. In this case, using again \eqref{DmildY} we have
\begin{align*}
\norm{\J_GX(t,x)h}_{C}\leq \|e^{tA_C}h\|_C+\int^t_0\|e^{(t-s)A_C}C^{1/2}\J G(X(s,x))\J_G X(s,x)h\|_Cds
\end{align*}
for any $t>0, x\in \X$ and $h \in H_C$. Thus, the Lipschitz continuity of $G$, the contractivity of $\{e^{tA_C}\}_{t\geq 0}$ in $H_C$ and \eqref{stigraX} yield that
\begin{align*}
\|\J_GX(t,x)h\|_{C}&\leq \|e^{tA_C}h\|_C+\int^t_0\|e^{(t-s)A_C}C^{1/2}\J G(X(s,x))\J_G X(s,x)h\|_Cds\\
&\leq \|h\|_C+\int^t_0\|C^{1/2}\J G(X(s,x))\J_G X(s,x)h\|_Cds\\
&=\|h\|_C+\int^t_0\|\J G(X(s,x))\J_G X(s,x)h\|ds\\
&\leq \|h\|_C+L_G\int^t_0\|\J_G X(s,x)h\|ds\\
&\leq\|h\|_C+L_G\|C^{1/2}\|_{\mathcal{L}(\X)}\|h\|_C\int^t_0e^{-\zeta_\X s}ds\\
&=\|h\|_C+L_G\|C^{1/2}\|_{\mathcal{L}(\X)}\|h\|_C\pa{\frac{1}{\zeta_\X}-\frac{e^{-\zeta_\X t}}{\zeta_\X}}\\
&\leq \pa{1+\frac{L_G\|C^{1/2}\|_{\mathcal{L}(\X)}}{\zeta_\X}}\|h\|_C,
\end{align*}
whence the claim follows, with $K_1:=1+L_G\zeta_\X^{-1}\|C^{1/2}\|_{\mathcal{L}(\X)}$.
\end{proof}

\begin{rmk}
We could remove the requirement $\zeta_\X>0$ in Hypotheses \ref{HA} and replace the condition \eqref{case1} with the following: there exist $c>0$ and $\gamma\in(0,1)$ such that
\begin{align*}
e^{tA}(\X)\subseteq C^{1/2}(\X),\qquad \|C^{-1/2}e^{tA}\|_{\mathcal{L}(\X)}\leq K e^{ct} t^{-\gamma}.
\end{align*}
In this case \eqref{sti} would be changed as follows: for any $T>0$ there exist $K_1(T)$, such that for any $t>0$, $x \in\X$ and $h\in H_C$ we have
\begin{align*}
\|\J_GX(t,x)h\|_{C}\leq K_{1}(T)\norm{h}_{C}.
\end{align*}
\end{rmk}

Now, we prove that for any fixed $t>0$, the map $\X \ni x \mapsto X(t,x)\in \X$ is $H_C$-Gateaux differentiable.

\begin{prop}
Assume that Hypotheses \ref{HA} and \ref{Sinem} hold true. For any $t>0$, the map $x\mapsto X(t,x)$, from $\X$ to $\X$, is $H_{C}$-Gateaux differentiable  $\mathbb{P}$-a.e. and its $H_C$-Gateaux derivative along $h\in H_{C}$ is given by $\J_G X(t,x)h$.
\end{prop}

\begin{proof}
We restrict ourselves to prove the claim when Hypotheses \ref{Sinem}\eqref{Luz} hold true since the other case can be obtain similarly arguing as in Proposition \ref{stimild1}. First of all we prove that for any fixed $t>0$, $x\in\X$ and $h\in H_C$, the function $\varphi^t_{x,h}:\R\ra\X$ defined as
\[\varphi_{x,h}^t(r):=X(t,x+rh)-X(t,x),\]
is $H_C$-valued. Using \eqref{dermildX} we obtain
\begin{align*}
\varphi_{x,h}^t(r)=re^{tA}h+\int_0^te^{(t-s)A}\big[F(X(s,x+rh))-F(X(s,x))\big]ds,\qquad r>0.
\end{align*}
Thus, Hypotheses \ref{Sinem}\eqref{Luz}, the definition of $\norm{\cdot}_C$, the Lipschitz continuity of $F$ and estimate \eqref{lipdete} infer
\begin{align*}
\|\varphi_{x,h}^t(r)\|_C &\leq \|re^{tA}h\|_C+\int_0^t\big\|e^{(t-s)A}\big[F(X(s,x+rh))-F(X(s,x))\big]\big\|_Cds\\
&\leq \|rC^{-1/2}e^{tA}h\|+\int_0^t\big\|C^{-1/2}e^{(t-s)A}\big[F(X(s,x+rh))-F(X(s,x))\big]\big\|ds\\
&\leq K|r|t^{-\gamma}\|h\|+\int_0^t\|C^{-1/2}e^{(t-s)A}\|_{\mathcal{L}(\X)}\|F(X(s,x+rh))-F(X(s,x))\|ds\\
&\leq  K|r|t^{-\gamma}\|h\|+KL_F\int_0^t (t-s)^{-\gamma}\|X(s,x+rh)-X(s,x)\|ds\\
&\leq  K|r|t^{-\gamma}\|h\|+KL_F|r|\|h\|\int_0^t (t-s)^{-\gamma}e^{-\eta s}ds\\
&\leq  K|r|\|h\|\left(t^{-\gamma}+L_F\frac{t^{1-\gamma}}{1-\gamma}\right)<+\infty.
\end{align*}
So $\varphi_{x,h}^t$ is $H_C$-valued.

Thanks to Proposition \ref{stimild1}, to conclude we just need to prove \eqref{lim_dif} with $\varphi_{x,h}$ being replaced by $\varphi_{x,h}^t$. For any $x\in\X$, $h\in H_C$ and $t,r>0$, by \eqref{DmildY} and \eqref{case1} we have
\begin{align*}
&\E\left[\norm{\frac{1}{r}\varphi_{x,h}^t(r)-\J_G X(t,x)h}_{C}\right]\\
&= \E\left[\norm{\int^t_0 C^{-1/2}e^{(t-s)A}\left(\frac{F(X(s,x+rh))-F(X(s,x))}{r}-\J F(X(s,x))\J_GX(s,x)h\right)ds}\right]\\
&\leq K\int^t_0 \frac{1}{(t-s)^{\gamma}}\E\left[\norm{\frac{F(X(s,x+rh))-F(X(s,x))}{r}-\J F(X(s,x))\J_GX(s,x)h}\right]ds.
\end{align*}
By \eqref{lipdete}, Proposition \ref{prodiff}, the fact that $\gamma\in(0,1)$ and the fact that $F$ is Fr\'echet differentiable and Lipschitz continuous we can apply the dominated convergence theorem to deduce that
\[
\lim_{r\ra 0}\E\left[\norm{\frac{1}{r}\varphi_{x,h}(r)-\J_G X(t,x)h}_{C}\right]=0,
\]
which concludes the proof.
\end{proof}

The following proposition is a chain rule for $H_C$-differentiable functions.

\begin{prop}\label{chain}
Assume Hypotheses \ref{HA} and \ref{Sinem} hold true. If $G:\X\ra\R$ is a Fr\'echet differentiable function, then $G$ is $H_C$-Fr\'echet differentiable and for any $t\geq 0$, $x\in\X$ and $h\in H_{C}$
\begin{align*}
[ \J_C \left(G\circ X(t,\cdot)\right)(x),h]_{C}=[ \J_C G(X(t,x)),\J_G X(t,x)h]_C.
\end{align*}
\end{prop}

\begin{proof}
The proof follows the same ideas of \cite[Corollary 21]{BI-FE1} with some minor changes.
\end{proof}

A key tool to prove the (LHI) stated in Theorem \ref{lip_harnak} is a finite-dimensional approximation procedure which allows us to approximate the transition semigroup $\{P(t)\}_{t\geq 0}$ by means of a sequence of transition semigroups $\{P^n(t)\}_{t\geq 0}$ associated to suitable finite-dimensional stochastic differential equations. The idea of such approximation comes from \cite{DaPROWA09}. Here, for the sake of completeness and to point out the minimal assumptions needed for such kind of procedure, we recall it and we provide a proof of the main approximation result (Proposition \ref{fin_approx}).

We need one more condition.
\begin{hyp}\label{xn}
Assume that $H_C$ is dense in $\X$ and that there exists a sequence of $A$-invariant and $C$-invariant finite dimensional subspaces $\X_n\subseteq {\rm Dom}(A_C)$ such that $\bigcup_{n=1}^{\infty}\X_n$ is dense in  $H_C$.
\end{hyp}
\noindent Hypotheses \ref{xn} hold true for instance, if $A$ is a self-adjoint positive operator and $C$ admits a continuous inverse $C^{-1}\in \mathcal{L}(X)$ or if $A$ and $C$ are simultaneously diagonalizable.

In view of Hypothesis \ref{xn} we can consider $\{e_k\}_{k\in\N} \in {\rm Dom}(A_C)$ such that for any $n\in\N$
\[\X_n={\rm span}\{e_1,\dots, e_n\},\]
and the family $\{e_k\,|\, k \in \N\}$ is an orthonormal basis of $H_C$. Further, let $\pi_n:\X\ra\X_n$ be the orthogonal projection with respect to $(\X, \langle \cdot, \cdot\rangle)$, for any $n \in \N$ we can define
$A_n:\X\ra\X_n$, $C_n:\X\ra\X_n$ and $F_n:\X\ra\X_n$ as
\begin{align*}
A_n:=\pi_nA\pi_n(=A \pi_n) ,\qquad C_n:=\pi_nC\pi_n(=C \pi_n) \qquad\text{ and }\qquad F_n:=\pi_nF\pi_n
\end{align*}
Now, fix $n \in \N$ and consider
\begin{equation}\label{SDE_n}
\left\{\begin{array}{ll}
dX_n(t)=[A_nX_n(t)+F_n(X_n(t))]dt+ C_n^{1/2}dW_n(t),&  t> 0,\\
X_n(0)=x\in \X_n.
\end{array}\right.
\end{equation}
Here $W_n(t):=\pi_nW(t)= \sum_{k=1}^n \langle W(t),e_k\rangle e_k$.

It is straightforward to see that $A_n, C_n$ and $F_n$ satisfy Hypotheses \ref{HA}. Moreover, being $C$ an injective operator it follows that $C_n$ is bijective hence $C^{1/2}_n\X_n=\X_n$ for any $n \in \N$. Therefore, fixed $x \in \X_n$, by Theorem \ref{Genmild} we can deduce existence and uniqueness of a mild solution $\{X_n(t,x)\}_{t\geq 0}$ of \eqref{SDE_n}
and consequently well-posedness for the associated transition semigroup defined for $f\in B_b(\X_n)$ as
\begin{align}\label{Sem_fin}
P_n(t)f(x):=\mathbb{E}[f(X_n(t,x))],\qquad t>0,\ x\in \X_n.
\end{align}
we recall that the map $(t,x)\mapsto (P_n(t)f)(x)$ solves the Cauchy problem associated to the second order elliptic operator $\mathcal{N}_n$ acting on cylindrical smooth functions as follows
\begin{equation*}
\mathcal{N}_n \varphi(x)= \frac{1}{2}{\rm Tr}[C_n D^2 \varphi(x)]+ \langle A_n x+ F_n(x), D\varphi(x)\rangle, \qquad\ x \in \X_n,\ n\in \N.
\end{equation*}

Now we are able to state the main finite-dimensional approximation result. The proof of this result is the same of \cite[Proposition 3.1(i)]{DaPROWA09}, with some minor changes. We decided to give it here for the sake of completeness.
\begin{prop}\label{fin_approx}
Assume that Hypotheses \ref{HA} and \ref{xn} hold true. For any $f\in C_b(\X)$, $t\geq 0$ and $x\in \X_{n_0}$, for some $n_0\in\N$, it holds
\begin{align*}
\lim_{n\ra+\infty}P_n(t)f(x)=P(t)f(x).
\end{align*}
\end{prop}

\begin{proof}
Let $\{X(t,x)\}_{t\geq 0}$ be the unique mild solution of \eqref{SDE}. For any $t\geq 0$, we set $Z(t):=X(t)-W_{A}(t)$, where $W_A(t)$ is defined in \eqref{stoconv}. For any fixed $x\in \X_{n_0}$ the process $\{Z(t)\}_{t\geq 0}$ is the unique mild solution of the problem
\begin{align*}
\eqsys{dZ(t)=(AZ(t)+F(Z(t)+W_A(t)))dt, & t>0;\\
Z(0)=x\in\X_{n_0},}
\end{align*}
which satisfies
\begin{align*}
\mathbb{E}\sq{\sup_{t\in[0,T]}\|Z(t)\|^2}<+\infty
\end{align*}
Since
\begin{align*}
\mathbb{E}\sq{\sup_{t\in[0,T]}\|W_A(t)\|^2}<+\infty,
\end{align*}
by the dominated convergence theorem, it is easy to see that $\pi_n W_A(t)$ converges to $W_A(t)$ in $L^2(\Omega,\mathbb{P})$, as $n$ tends to infinity.

Setting $W_n(t):=\pi_n W(t)$ and
\begin{align*}
W_{A_{n}}(t):=\int_0^{t}e^{(t-s)A_n}C_n^{1/2}dW_n(s),\qquad t\geq 0,
\end{align*}
then the process $Z_n(t):=X_n(t,x)-W_{A_{n}}(t)$, ($n\ge n_0$) satisfies
\begin{align*}
\eqsys{dZ_n(t)=(A_nZ_n(t)+F_n(Z_n(t)+W_{A_{n}}(t)))dt, & t>0;\\
Z_n(0)=x\in\X_{n_0}.}
\end{align*}
Now we split the proof in two steps. In the first one we show that $W_{A_n}(t)-W_{A}(t)$ converges to $0$ in $L^2(\omega, \mathbb{P})$ and that, consequently $\|Z_n(t)\|$ is uniformly bounded with respect to $n$. In the second one we complete the proof.\\
{\emph Step 1.} Using that $e^{tA}_{|_{\X_n}}=e^{tA_n}$ for every $t\geq0$ we can write
\begin{align*}
W_A(s)-W_{A_n}(s)=\int_0^se^{(s-r)A}\pa{C^{1/2}-\pi_nC^{1/2}\pi_n}dW(r), \qquad\,\, s \in [0,T]
\end{align*}
By the It\^o formula we deduce
\begin{align*}
\mathbb{E}(\|W_A(s)-W_{A_n}(s)\|^2)&=\int_0^s\norm{e^{(s-r)A}\pa{C^{1/2}-\pi_nC^{1/2}\pi_n}}_{\mathcal{L}(\X)}^2dr
\end{align*}
and, since the integrand converges to zero as $n \to \infty$ uniformly with respect to $r\in (0,s)$, by the dominated convergence theorem we get the claim.
Now, scalarly multiplying $dZ_n(t)=(A_nZ_n(t)+F_n(Z_n(t)+W_{A_{n}}(t)))dt$ by $Z_n(t)$ and using the Hypotheses \ref{HA} with $A$ and $F$ being replaced by $A_n$ and $F_n$ we deduce (see \eqref{noname})
\begin{align*}
\frac{1}{2}\frac{d}{dt}\|Z_n(t)\|^2 =& \langle A_n Z_n(t)+ F_n(Z_n(t)+W_{A_n}(t)), Z_n(t)\rangle\\
= &\langle A_n Z_n(t)+ F_n(Z_n(t)+W_{A_n}(t))\pm F_n(W_{A_n}(t)), Z_n(t)\rangle\\
\le & \zeta_{\X}\|Z_n(t)\|^2+ \langle F_n(W_{A_n}(t)), Z_n(t)\rangle\\
\le & \left(\zeta_{\X}+\frac{1}{2}\right)\|Z_n(t)\|^2+ \frac{1}{2}\| F_n(W_{A_n}(t))\|^2\\
\le & \left(\zeta_{\X}+\frac{1}{2}\right)\|Z_n(t)\|^2+ M2^{-1}\left(1+\|W_{A_n}(t)\|^2 \right), \qquad\,\, t>0.
\end{align*}
The Gronwall lemma and the uniform boundedness of $\|W_{A_n}(t)\|$ with respect to $n$ allows to deduce that $\|Z_n(t)\|$ is uniformly bounded with respect to $n$.\\
{\emph Step 2.}
To conclude the proof we show that $Z_n(t)-Z(t)$ converges to $0$ in $L^2(\Omega,\mathbb{P})$ as $n\to \infty$.
This fact will imply that $X_n(t,x)$ converges to $X(t,x)$ in $L^2(\Omega,\mathbb{P})$ as $n\to\infty$ and by \eqref{trans} and \eqref{Sem_fin} we conclude.
By Hypothesis \ref{xn} we have
\begin{align*}
d(Z(t)-Z_n(t))&=(AZ(t)-A_nZ_n(t)+F(X(t,x))-F_n(X_n(t,x)))dt\\
&=(A(Z(t)-Z_n(t))+F(X(t,x))-F_n(X_n(t,x)))dt,\qquad t>0;
\end{align*}
so, scalarly  multiplying by $Z(t)-Z_n(t)$ we have
\begin{align*}
\frac{1}{2}\frac{d}{dt}\norm{Z(t)-Z_n(t)}^2 =&\langle (A(Z(t)-Z_n(t))+F(X(t,x))-F_n(X_n(t,x))), Z(t)-Z_n(t)\rangle\\
= &\langle (A(Z(t)-Z_n(t))+F(Z(t)+W_A(t))\\
&\qquad \qquad \qquad \quad\ \  -F_n(Z_n(t)+W_{A_n}(t))), Z(t)-Z_n(t)\rangle.
\end{align*}
Adding and subtracting the terms $F(Z_n(t)+W_A(t))$ and $F(Z_n(t)+W_{A_n}(t))$ and using Hypotheses \ref{HA}
we obtain
\begin{align*}
\frac{1}{2}\frac{d}{dt}\norm{Z(t)-Z_n(t)}^2& \leq \zeta_{\X}\norm{Z(t)-Z_n(t)}^2\\
&\phantom{aaaaaa}+\langle F(Z_n(t)+W_A(t))-F(Z_n(t)+W_{A_n}(t)), Z(t)-Z_n(t)\rangle\\
&\phantom{aaaaaa}+\langle F(Z_n(t)+W_{A_n}(t))-\pi_n F(Z_n(t)+W_{A_n}(t)), Z(t)-Z_n(t)\rangle\\
&\leq \zeta_{\X}\norm{Z(t)-Z_n(t)}^2\\
&\phantom{aaaaaa}+\|F(Z_n(t)+W_A(t))-F(Z_n(t)+W_n(t))\|\|Z(t)-Z_n(t)\|\\
&\phantom{aaaaaa}+\| ({\rm Id}_{\X}-\pi_n)F(Z_n(t)+W_{A_n}(t))\|\|Z(t)-Z_n(t)\|\\
&\leq \zeta_{\X}\norm{Z(t)-Z_n(t)}^2+L_F\|W_A(t)-W_{A_n}(t)\|\|Z(t)-Z_n(t)\|\\
&\phantom{aaaaaa}+\frac{1}{2}\| ({\rm Id}_{\X}-\pi_n)F(Z_n(t)+W_{A_n}(t))\|^2+\frac{1}{2}\|Z(t)-Z_n(t)\|^2\\
&\leq (\zeta_{\X}+1)\norm{Z(t)-Z_n(t)}^2+\frac{1}{2}L_F^2\|W_A(t)-W_{A_n}(t)\|^2\\
&\phantom{aaaaaa}+\frac{1}{2}\| ({\rm Id}_{\X}-\pi_n)F(Z_n(t)+W_{A_n}(t))\|^2.
\end{align*}
where in the last two lines we have used the Young inequality.
\begin{align*}
\frac{1}{2}\|&Z(t)-Z_n(t)\|^2\\
&\leq \int_0^t \pa{\zeta_\X+\frac{1}{2}}\|Z(s)-Z_n(s)\|^2+\|F(X(s))-F(X_n(s))\|^2+\|(1-\pi_n)F(X_n(s))\|^2ds\\
&\leq  \int_0^t \pa{\zeta_\X+\frac{1}{2}+2L_F^2}\|Z(s)-Z_n(s)\|^2+2L_F^2\|W_A(s)-W_n(s)\|^2+\|(1-\pi_n)F(X_n(s))\|^2ds.
\end{align*}
Applying again the Gronwall lemma we obtain
 \begin{align}\label{pipo}
 \norm{Z(t)-Z_n(t)}^2\le &L_F^2e^{2(\zeta_{\X}+1)t}\int_0^t\|W_A(s)-W_{A_n}(s)\|^2ds\notag\\
 &+e^{2(\zeta_{\X}+1)t}\int_0^t\| ({\rm Id}_{\X}-\pi_n)F(Z_n(s)+W_{A_n}(s))\|^2ds
 \end{align}
and using the results in Step 1 we infer that the right hand side of \eqref{pipo} vanishes as $n \to \infty$, concluding the proof.
\end{proof}

We point out that the assumption of $C$-invariance of $\X_n$ can be dropped in order to prove Proposition \ref{fin_approx}. However it is essential to apply the results in the previous sections to the finite-dimensional approximating operators. Indeed, assuming Hypotheses \ref{xn}, it is easy to see that if $A,C$ and $F$ satisfy Hypothesis \ref{Sinem}(i) then $A_n, C_n$ and $F_n$ satisfy them as well with the same constants. So all the results of the previous sections hold true even for the mild solution of \eqref{SDE_n} and for the semigroup in \eqref{Sem_fin}. In particular for every $h\in H_{C_n}$
\begin{align}\label{stifin}
\|\J_GX_n(t,x)h\|_{C_n}\leq K_1\norm{h}_{C_n},\qquad\text{for }x\in\X_n\text{ and }\mathbb{P}\text{-a.e.},
\end{align}
where $K_{1}$ is the same constant appearing in Proposition \ref{stimild1}.

The next result is a gradient estimate which is interesting in its own right. For our aims it is fundamental to prove Theorem \ref{lip_harnak}.
\begin{thm}
Assume Hypotheses \ref{HA}, \ref{Sinem} and \ref{xn} hold true. For any $\varphi\in C^1_b(\X)$, $P(t)\varphi$ is Fr\'echet differentiable and for any $t>0$
\begin{align}\label{V1}
\|\J_C P(t)\varphi(x)\|_C\leq K_1P(t)\|\J_C \varphi(x)\|_C, \qquad x \in \X;
\end{align}
and
\begin{align}\label{OOC}
\|\J_{C_n} P_{n}(t)\varphi(x)\|_{C_n}\leq K_1P(t)\|\J_{C_n} \varphi(x)\|_{C_n}, \qquad x \in \X_n,\ n\in \N.
\end{align}
\end{thm}

\begin{proof}
We restrict ourselves to prove estimate \eqref{V1} since \eqref{OOC} can be obtained similarly using \eqref{stifin} in place of \eqref{sti}. By \eqref{stigraX}, Proposition \ref{chain} and \cite[Fact 1.13(b), p. 8]{PHE1} it holds that $P(t)\varphi$ is both Fr\'echet differentiable and $H_C$-Fr\'echet differentiable (see Remark \ref{sun}). Moreover, thanks to Propositions \ref{stimild1} and \ref{chain} we obtain
\begin{align*}
[\J_C P(t)\varphi(x),h]_C &=[\J_C\E[\varphi(X(t,x))],h]_C\\
&=\E\big[[\J_C\varphi(X(t,x)),\J_GX(t,x)h]_C\big]\\
&\leq \E\big[\|\J_C\varphi(X(t,x))\|_C\|\J_GX(t,x)h\|_C\big]\\
&\leq K_1\|h\|_C P(t)\|\J_C\varphi(x)\|_C,
\end{align*}
for any $t\geq 0$, $x\in\X$ and $h\in H_{C}$. Now \eqref{V1} follows by a standard argument.
\end{proof}

We are ready to prove the main theorem of this section. 

\begin{thm}\label{lip_harnak}
Assume  Hypotheses \ref{HA}, \ref{Sinem} and \ref{xn} hold true. For any $\varphi\in B_b(\X)$ and $p>1$ it holds
\begin{align}\label{lip_harnak_est}
|P(t)\varphi(x+h)|^p\leq P(t)|\varphi(x)|^p\,{\rm exp}\pa{\frac{pK_1^2}{t(p-1)}\|h\|_C^2},\qquad t>0,\ x \in \X,\ h \in H_C.
\end{align}
Here $K_1$ is the constant appearing in Proposition \ref{stimild1}.
\end{thm}

\begin{proof}
Fix $n, n_0\in \N$ with $n>n_0$ and $x, h \in \X_{n_0}$. We claim that estimate \eqref{lip_harnak_est} holds true with $P(t)$, $C$, $\X$ and $\varphi$ replaced respectively by $P_n(t)$, $C_n$, $\X_n$ and $f$, being $f \in \mathcal{F}C^2_b(\X)$ with positive infimum. To this aim, fix $\varepsilon>0$ and $f \in \mathcal{F}C^2_b(\X)$ be such that $f(x)\geq \eps$ for any $x\in\X$. We set
\[
g_{n,n_0}(r,x):=P_{n}(r)f(x),\qquad x\in\X_{n_0},\ r\geq 0,
\]
and we note that for any $r>0$, the function $g_{n,n_0}:[0,T]\times\X_{n_0}\ra \R$ belongs to $C^{1,2}([0,T]\times\X_{n_0})$ solves
\begin{gather*}
\eqsys{
D_rg_{n,n_0}(r,x)=\mathcal{N}_{n}g_{n,n_0}(r,x), & r>0,\ x\in\X_{n_0};\\
g_{n,n_0}(0,x)=f(x), & x\in\X_{n_0}.
}
\end{gather*}
and satisfies $g_{n,n_0}(r,x)\geq \eps$ for any $r\geq 0$ and $x\in\X_{n_0}$ (see \cite[Theorem 1.2.5]{lorbook}).
Now fix $t>0$ and consider the function
\[
G_n(r):=\left(P_{n}(t-r)g_{n,n_0}^p(r,\cdot)\right)\pa{x+rt^{-1}h},\qquad r\in [0,t],\ x,h\in\X_{n_0}.
\]
For the sake of simplicity we let $\psi_h(r):=x+rt^{-1}h$. We differentiate the map $r\mapsto\ln(G(r))$.
\begin{align}
\frac{d}{dr}\ln G_n(r)=&\,(G_n(r))^{-1}P_n(t-r)\Big(-\mathcal{N}_{n}g_{n,n_0}^p(r,\cdot)+D_r g_{n,n_0}^p(r,\cdot)\Big)(\psi_h(r))\notag\\
&+(t G_n(r))^{-1}\langle \J P_n(t-r)g_{n,n_0}^p(\psi_h(r)),h\rangle.\label{H0}
\end{align}
where we used that the semigroup and its generator commute on smooth functions.
A straightforward computation yields that
\begin{equation}\label{d-a}
-\mathcal{N}_{n}g_{n,n_0}^p+D_r g_{n,n_0}^p=-p(p-1)g^{p-2}_{n,n_0}\|C_n^{1/2}\J g_{n,n_0}\|_{\X_{n}}^2
\end{equation}
whereas by estimate \eqref{OOC} we infer
\begin{align}\label{gra-est}
\langle\J & P_n(t-r)g_{n,n_0}^p(\psi_h(r)),h\rangle\notag \\
&=\langle C_n^{1/2}\J P_n(t-r)g_{n,n_0}^p(\psi_h(r)), C_n^{-1/2}h\rangle \notag\\
&\leq K_1\|C_n^{-1/2}h\|_{\X_{n}} P_n(t-r)\sq{pg_{n,n_0}^{p-1}(r,\cdot)\|C_n^{1/2}\J g_{n,n_0}(r,\cdot)\|_{\X_{n}}}(\psi_h(r)).
\end{align}
Hence, using \eqref{d-a} and \eqref{gra-est} in \eqref{H0} we obtain
\begin{align*}
\frac{d}{dr}\ln G_n(r)\leq (G_n(r))^{-1}P_n(t-r)[-p(p-1)g_{n,n_0}^{p}(r,\cdot)(-v^2+\beta v)](\psi_h(r))
\end{align*}
where
\[
v:=g_{n,n_0}^{-1}(r,\psi_h(r))\|C_n^{1/2}\J g_{n,n_0}(r,\psi_h(r))\|_{\X_{n}},\qquad \beta:=\frac{K_{1}}{t(p-1)}\|C_n^{-1/2}h\|_{\X_{n}}.
\]
The elementary inequality $a^2-ab+4^{-1}b^2\ge 0$ which holds true for any $a,b\in\R$ allows us to estimate $-v^2 +\beta v\leq \beta^2$ and, consequently, by the positivity of $P_n(t)$ (see again \cite[Theorem 1.2.5]{lorbook}), to deduce that
\begin{align*}
\frac{d}{dr}\ln G_n(r)\leq
 \frac{pK_{1}^2}{t^2(p-1)}\|C_n^{-1/2}h\|_{\X_{n}}^2.
\end{align*}
whence, integrating from $0$ to $t$ with respect to $r$ we obtain the claim.

Now, by a standard approximation argument, the Jensen inequality, Proposition \ref{fin_approx} and the fact that $\|h\|_{{C_n}}$ converges to $\|h\|_{{C}}$ for any $h \in H_C$ as $n \to \infty$, we get
\begin{align}\label{chill}
|P(t)f(x+h)|^p\leq P(t)|f(x)|^p\,{\rm exp}\pa{\frac{pK_{1}^2}{t(p-1)}\|h\|_{H_C}^2},\qquad t>0,\
\end{align}
for any $f\in \mathcal{F}C_b^2(\X)$ and $x,h\in\bigcup_{i\in\N}\X_{i}$ (recall that $C_n^{1/2}\X_n=\X_n$ for any $n\in \N$).

A double approximation argument allows to extend \eqref{chill} first to bounded and continuous functions and then to Borel bounded functions (see \cite[Lemma 2.6 and Theorem 4.1(b)]{GOKO01} and \cite[Theorem 6.3]{JACPRO03}).

Finally, let us fix $x\in\X$ and $h\in H_C$, by Hypothesis \ref{xn} we can find two sequences $(x_n)_{n\in\N}$ and $(h_m)_{m\in\N}$ belonging to $\bigcup_{i\in\N}\X_{i}$ converging respectively to $x$ in $\X$ and $h$ in $H_C$ as $n, m\to \infty$ (recall that, by Hypotheses \ref{xn}, $H_C$ is dense in $\X$).
Writing \eqref{chill} with $x$ and $h$ being replaced by $x_n$ and $h_m$ we deduce
\begin{align}\label{Iren}
|P(t)f(x_n+h_m)|^p\leq P(t)|f(x_n)|^p\,{\rm exp}\pa{\frac{pK_{1}^2}{t(p-1)}\|h_m\|_C^2},\qquad t>0,\ f\in B_b(\X).
\end{align}
By the continuity of the map $x\mapsto P(t)f(x)$, we complete the proof, letting $n, m$ tend to $\infty$ in \eqref{Iren}.
\end{proof}


\section{Logarithmic Harnack inequality: the dissipative case}\label{nonome}

The aim of this section is proving some (LHI) when no hypotheses of global Lipschitzianity for $F$ is done but only some $m$-dissipativity along $H_C$. The main tool is again an approximation procedure this time using the Yosida approximants, 

The standing assumptions we consider here are Hypotheses \ref{EU2} which allows us to deal with the mild solution to  \eqref{SDE} with initial datum $x \in E$ and also guarantees the existence of two positive constants $C_p$ and $\kappa_p$ such that for every $x\in E$ and $t\geq 0$
\begin{align}\label{patatine}
\norm{X(t,x)}_E^p
\leq  C_p\left(e^{-\kappa_p t}\norm{x}^p_E+\norm{W_A(t)}_E^p+\int_0^te^{-\kappa_p (t-s)}\norm{F(W_A(s))}^p_Eds\right),\qquad \mathbb{P}\text{-a.e.}
\end{align}
Consequently, the map $(t,x)\mapsto \norm{X(t,x)}_E$ is bounded as a function from $[0,+\infty)\times E$ to $\R$.

We start by recalling the definition of the Yosida appoximants for the function $F$. For any $\delta>0$ and $x\in\X$ we let $J_\delta(x)\in \Dom(F)$ be the unique solution of
\[y-\delta (F(y)-\zeta_F y)=x.\]
The existence of $J_\delta(x)$, for every $x\in \X$ and $\delta>0$, is guaranteed by \cite[Proposition 5.3.3]{DA-ZA2}. We define $F_\delta:\X\ra\X$ as
\begin{align*}
F_\delta(x):=F(J_\delta(x)),\qquad x\in\X,\ \delta>0.
\end{align*}

\begin{rmk}\label{lisa}{\rm
We point out that $\Dom(F_\delta)=\X$ and that Hypotheses \ref{EU2} guarantee that, for any $\delta>0$ and $x \in E$, $J_\delta(x)$ belongs to $E$, too. In particular, since $E$ is preserved by the action of $F$, it turns out that $F_\delta(E)\subseteq E$.}
\end{rmk}

Now, we set
\begin{equation*}
\overline{\zeta}_F:=\left\{\begin{array}{ll}|\zeta_F|^{-1} \qquad\,\,&{\rm if}\,\,\zeta_F\neq 0\\
+\infty\qquad\;\,&{\rm if}\,\,\zeta_F=0
\end{array}\right.
\end{equation*}
and we list some useful properties of the Yosida approximants.

\begin{lemm}\label{Lemma_YO}
If Hypothesis \ref{EU2}(iii) holds true, then
\begin{align}
&\lim_{\delta\ra 0}\norm{J_{\delta}(y)-y}_E=0,& &\lim_{\delta\ra 0}\|(F_{\delta})_{|_E}(y)-F_{|_E}(y)\|_E=0,& &y\in E.\label{conveyos}
\end{align}
For any $\delta \in (0,\overline{\zeta}_F)$, the functions $F_\delta-\zeta_F\Id_{\X}$ and $(F_\delta)_{|_E}-\zeta_F\Id_E$ are $m$-dissipative on $\X$ and $E$, respectively. Moreover for any $\delta>0$ it hold
\begin{align}
\norm{J_\delta(y)-y}_E&\leq \delta(M+M\|y\|^m_E+\zeta_F\|y\|_E), &y\in E;\label{cji}\\
\|(F_{\delta})_{|_E}(y)\|_E&\leq(3+\delta\zeta_F)\|F_{|_E}(y)\|+(2\zeta_F+\delta \zeta_F^2)\norm{y}_E. &y\in E;\label{vy1}
\end{align}
and
\begin{align}
\norm{F_\delta(x_1)-F_\delta(x_2)}\leq &\left(\frac{2}{\delta}+\zeta_F\right)\norm{x_1-x_2},\qquad x_1,x_2\in\X;\label{lipdeltaX}\\
\|(F_{\delta})_{|_E}(y_1)-(F_{\delta})_{|_E}(y_2)\|_{E
}\leq &\left(\frac{2}{\delta}+\zeta_F\right)\norm{y_1-y_2}_{E},\qquad y_1,y_2\in E.\label{lipdeltaE}
\end{align}

\end{lemm}

\begin{proof}
By applying \cite[Proposition 5.5.3]{DA-ZA2} with $f=F-\zeta_F{\rm Id}_{\X}$ and $f=F_{|_E}-\zeta_F{\rm Id}_E$ we get immediately \eqref{conveyos}, the $m$-dissipativity of $F_\delta-\zeta_F{\rm Id}_{\X}$ and $(F_\delta)_{|_E}-\zeta_F{\rm Id}_E$ and estimates \eqref{lipdeltaX} and \eqref{lipdeltaE}.
Moreover \cite[Proposition A.2.2(4)]{CER1} yields that
\begin{equation}\label{cerrai}
\|J_\delta(y)-y\|_E\le \delta \|F_{|E}(y)-\zeta_F y\|_E, \qquad\;\, y\in E, \delta >0.
\end{equation}
Estimate \eqref{cji} thus follows by \eqref{noname} and \eqref{cerrai}.
Hence it remains to prove \eqref{vy1}. By \cite[Proposition 5.5.3(ii)]{DA-ZA2} we can estimate
$\|(F_{\delta})_{|_E}(y)-\zeta_F J_\delta(y)\|_E\le \|F_{|_E}(y)-\zeta_F y\|_E$ for any $y\in E$ and $\delta>0$. Using this fact together with \eqref{cerrai} we get
\begin{align*}
\|(F_{\delta})_{|_E}(y)\|_E &\leq \|F_{|_E}(y)\|_E+ \|(F_{\delta})_{|_E}(y)-F_{|_E}(y)\|_E\\
&\leq \|F_{|_E}(y)\|_E+\|(F_{\delta})_{|_E}(y)-\zeta_F J_\delta(y)\|_E+\|F_{|_E}(y)-\zeta_F y\|_E+\zeta_F\norm{J_\delta(y)-y}_E\\
&\leq \|F_{|_E}(y)\|_E+2\|F_{|_E}(y)-\zeta_F y\|_E+\delta\zeta_F\norm{F_{|_E}(y)-\zeta_F y}_E\\
&\leq (3+\delta\zeta_F)\|F_{|_E}(y)\|+(2\zeta_F+\delta \zeta_F^2)\norm{y}_E
\end{align*}
for any $\delta>0$ and $y \in E$. Hence \eqref{vy1} is proved.
\end{proof}

Theorem \ref{Genmild}, Remark \ref{lisa} and Lemma \ref{Lemma_YO} imply that for every $\delta \in (0,1)$ and $x\in\X$ the problem
\begin{gather*}
\eqsys{
dX_\delta(t)=\big[AX_\delta(t)+F_\delta(X_\delta(t))\big]dt+C^{1/2}dW(t), & t>0;\\
X_\delta(0)=x\in \X
}
\end{gather*}
has a unique mild solution $\{X_\delta(t,x)\}_{t\geq 0}$ satisfying (see formula \eqref{vy1})
\begin{align}
&\norm{X_\delta(t,x)}_E^p
\leq  C_p\left(e^{-\kappa_p t}\norm{x}^p_E+\norm{W_A(t)}_E^p+\int_0^te^{-\kappa_p (t-s)}\left(\norm{F(W_A(s))}^p_E+ \norm{W_A(s)}^p_E\right)ds\right),\label{stindEzY}
\end{align}
for any $p\ge 1$, $t>0$, $x\in E$ and some positive constants $\kappa_p$ and $C_p=C_p(\zeta_F)$. Consequently the semigroup
\[P_\delta(t)\varphi(x):=\E[\varphi(X_\delta(t,x))]\]
is well defined for any $\varphi \in B_b(\X)$. Moreover, \eqref{cji} and \eqref{vy1} are crucial in order to prove, as the next proposition shows, that $P_\delta(t)$ approximates the semigroup $P(t)$ (see Theorem \ref{Genmild}) as $\delta \to 0$.

\begin{prop}
If Hypotheses \ref{EU2} hold true, then for any $T>0$, $\varphi\in C_b(\X)$ and $x\in E$
\begin{align}
\lim_{\delta\ra 0}\sup_{t\in [0,T]}\norm{X_\delta(t,x)-X(t,x)}_E&=0,\qquad \mathbb{P}\text{-a.e.};\label{convE}\\
\lim_{\delta\ra 0}\abs{P_\delta(t)\varphi(x)-P(t)\varphi(x)}&=0,\qquad t>0.\label{convE1}
\end{align}
\end{prop}

\begin{proof}
First of all let us observe that \eqref{convE1} immediately follows by \eqref{convE}. Therefore we just prove \eqref{convE}.
To this aim we start pointing out that \eqref{stindEzY} implies that the function $(t,x)\mapsto \norm{X_\delta(t,x)}_E$, as a map from $[0,T]\times E$ into $\R$, is bounded by a positive constant $C=C(T,x)$ independent of $\delta$.
This fact, together with estimates \eqref{patatine} and \eqref{cji} implies that
for any $\delta>0$, $x\in E$ and $T>0$
\begin{align}\label{findelta}
K(T, x):=\sup_{t\in [0,T]}\pa{\norm{J_\delta(X_\delta(t,x))}_E+\norm{X_\delta(t,x)}_E+\norm{X(t,x)}_E}<+\infty
\end{align}
and $K(T,x)$ is independent of $\delta$.
So by the local Lipschitzianity of $F$ (Hypotheses \ref{EU2}(iii)) there exists $L:=L(x,T)>0$ such that
\begin{align}
\norm{F_\delta(X_\delta(t,x))-F(X(t,x))}_E&=\norm{F(J_\delta(X_\delta(t,x)))-F(X(t,x))}_E\notag\\
&\leq L\norm{J_\delta(X_\delta(t,x))-X(t,x)}_E\notag\\
&\leq L\norm{J_\delta(X_\delta(t,x))-X_\delta(t,x)}_E+L\norm{X_\delta(t,x)-X(t,x)}_E.\label{papetee}
\end{align}
By \eqref{cji}, \eqref{findelta} and \eqref{papetee} we can conclude that
\begin{align}\label{boh}
\norm{F_\delta(X_\delta(t,x))-F(X(t,x))}_E&\leq \delta M'+L\norm{X_\delta(t,x)-X(t,x)}_E.
\end{align}
for some positive $M'=M'(K, M,L, m, \zeta_F)$.
Thus, by using the definition of mild solution and estimate \eqref{boh} we obtain
\begin{align*}
\norm{X_\delta(t,x)-X(t,x)}_E\leq \delta M_0M'\int^t_0e^{(t-s)\eta_0}ds+ M_0L\int^t_0e^{(t-s)\eta_0}\norm{X_\delta(t,x)-X(t,x)}_Eds.
\end{align*}
Applying the Gronwall lemma we complete the proof.
\end{proof}

As announced we need an additional assumption of dissipativity on $F$.

\begin{hyp}\label{Rick}It holds that:
\begin{enumerate}[\rm(i)]
\item  $F(\Dom(F)\cap H_{C})\subseteq H_{C}$ and $F_{|_{H_C}}-\zeta_F\Id_{H_C}$ is $m$-dissipative $(\zeta_F$ is the constant appearing in Hypothesis \ref{EU2}(iii)$)$.
\item
$A$ and $C$ are simultaneously diagonalizable and exists a basis $\{e_k\}_{k\in\N}$ of $H_C$ consisting of eigenvectors of $A$ and $C$.
\end{enumerate}
\end{hyp}

\begin{rmk}\label{diss-rem}{\rm
We need to point out two observations that we will use in this section.
\begin{enumerate}[\rm(i)]
\item
Arguing as in Lemma \ref{Lemma_YO} it is easy to prove that under Hypotheses \ref{EU2}(iii) and \ref{Rick}(i), the function $(F_{|H_C})_\delta-\zeta_F\Id_{|H_C}$ is dissipative as a function from $H_C$ to $H_C$ for any $\delta \in (0, \overline{\zeta}_F)$.
\item If Hypothesis \ref{Rick}(ii) holds true, then $H_C$ is dense in $\X$
and Hypotheses \ref{xn} are immediately satisfied with $\X_n:=\mbox{span}\{e_1,...,e_n\}$.
\end{enumerate}}
\end{rmk}

Under Hypotheses \ref{EU2}, \ref{Rick}, for any $\delta \in (0, \overline{\zeta}_F)$ and $n\in\N$ we let $\{P_{\delta,n}(t)\}
_{t\geq 0}$ be the transition semigroup associated to
\begin{equation}\label{SDE_ndelta}
\left\{\begin{array}{ll}
dX_{\delta,n}(t)=[A_nX_{\delta,n}(t)+(F_{\delta})_n(X_{\delta,n}(t))]dt+ C_n^{1/2}dW_n(t),&  t> 0;\\
X_{\delta,n}(0)=z\in \X_n.
\end{array}\right.
\end{equation}
where $A_n$, $C_n$ and $(F_{\delta})_n$ are defined in Section \ref{liphar}.

\begin{prop}\label{ROWA}
Under Hypotheses \ref{EU2} and \ref{Rick} the following estimate
\begin{align*}
\|\mathcal{D}_{C_n}P_{\delta,n}(t)f(x)\|_{C_n}\leq e^{t\zeta_{\X}}P_{\delta,n}(t)\|\mathcal{D}_{C_n}f(x)\|_{C_n}, \qquad\;\, x \in \X_n
\end{align*}
holds true for any $\delta \in (0, \overline{\zeta}_F)$, $n\in\N$, $t>0$, $f\in\mathcal{F}C^{1,n}_b(\X)$ and $\zeta_\X:=\zeta_A+\zeta_F$.
\end{prop}
\begin{proof}
We fix $n\in\N$, $\delta \in (0, \overline{\zeta}_F)$ and $x, y \in \X_n$ and let us consider $\{X_{n,\delta}(t,x)\}_{t\geq 0}$ and $\{X_{n,\delta}(t,y)\}_{t\geq 0}$ be the mild solutions of \eqref{SDE_ndelta} with initial datum $x$ and $y$ respectively. We assume that $\{X_{n,\delta}(t,x)\}_{t\geq 0}$ and $\{X_{n,\delta}(t,y)\}_{t\geq 0}$ are strict solutions of \eqref{SDE_ndelta}, otherwise we proceed as in \cite[Proposition 3.6]{BI1} or \cite[Proposition 6.2.2]{CER1} approximating them by means of a sequence of more regular processes. For any $t\geq 0$ we have
\begin{align*}
\frac{1}{2}\frac{d}{dt}\norm{X_{n,\delta}(t,x)-X_{n,\delta}(t,y)}^2_{C_n}&=\scal{A\left(X_{n,\delta}(t,x)-X_{n,\delta}(t,y)\right)}{X_{n,\delta}(t,x)-X_{n,\delta}(t,y)}_{C_n}\\
&+\scal{(F_\delta)_n(X_{n,\delta}(t,x))-(F_\delta)_n(X_{n,\delta}(t,y))}{X_{n,\delta}(t,x)-X_{n,\delta}(t,y)}_{C_n}.
\end{align*}
Since $\X_n$ is a finite dimensional space and the operators $A_n$ and $C_n$ are simultaneously diagonalizable, they commute. Thus, thanks to Hypotheses \ref{EU2} and \ref{Rick}(i) (see also Remark \ref{diss-rem}(i)), for any $t\geq 0$ we obtain
\begin{align*}
\frac{d}{dt}\norm{X_{n,\delta}(t,x)-X_{n,\delta}(t,y)}^2_{C_n}\leq -2\zeta_\X\norm{X_{n,\delta}(t,x)-X_{n,\delta}(t,y)}^2_{C_n},
\end{align*}
where $\zeta_\X=\zeta_A+\zeta_F$. By the Gronwall inequality, for any $t\geq 0$ and $x,y\in\X_n$ we obtain
\begin{align*}
\norm{X_{n,\delta}(t,x)-X_{n,\delta}(t,y)}_{C_n}\leq e^{-t\zeta_\X}\norm{x-y}_{C_n}.
\end{align*}
Therefore, for any $f\in\mathcal{F}C^{1,n}_b(\X)$, $t\geq 0$ and $x\in\X_n$ we have
\begin{align*}
\|\mathcal{D}_{C_n}P_{\delta,n}(t)f(x)\|_{C_n}&=\lim_{y\ra x}\dfrac{\abs{P_{\delta,n}(t)f(x)-P_{\delta,n}(t)f(y)}}{\norm{x-y}_{C_n}}\\
&=\lim_{y\ra x}\left(\dfrac{\abs{P_{\delta,n}(t)f(x)-P_{\delta,n}(t)f(y)}}{\norm{X(t,x)-Y(t,y)}_{C_n}}\right)\left(\dfrac{\norm{X(t,x)-Y(t,y)}_{C_n}}{\norm{x-y}_{C_n}}\right)\\
&= \lim_{y\ra x}\left(\dfrac{\abs{\E(f(X_{\delta,n}(t,x)))- \E(f(X_{\delta,n}(t,y)))}}{\norm{X(t,x)-Y(t,y)}_{C_n}}\right)\left(\dfrac{\norm{X(t,x)-Y(t,y)}_{C_n}}{\norm{x-y}_{C_n}}\right)\\
& \leq e^{-t\zeta_{\X}}P_{\delta,n}(t)\|\mathcal{D}_{C_n}f(x)\|_{C_n}
\end{align*}
whence the claim.
\end{proof}

\begin{thm}\label{Harnak_diss}
 Assume that Hypotheses \ref{EU2} and \ref{Rick} hold true. For any $\varphi\in B_b(\X)$,  $t>0$, $x\in\X$ and $p>1$ it holds
\begin{align}\label{dae}
|P(t)\varphi(x+h)|^p\leq P(t)|\varphi(x)|^p\,{\rm exp}\pa{\frac{pe^{2t\zeta_{\X}}}{t(p-1)}\|h\|_C^2},\qquad h\in H_C\cap E.
\end{align}
\end{thm}

\begin{proof}
Using Proposition \ref{ROWA} and arguing exactly as in the proof of Theorem \ref{lip_harnak} we obtain
\begin{align}\label{pdelta}
|P_\delta(t)\varphi(x+h)|^p\leq P_\delta(t)|\varphi(x)|^p\,{\rm exp}\pa{\frac{pe^{2t\zeta_{\X}}}{t(p-1)}\|h\|_C^2}.
\end{align}
for any $\delta \in (0,\overline{\zeta}_F)$, $p>1$, $t>0$, $\varphi \in C_b(\X)$, $x \in X$ and $h \in H_C$.
Now, using \eqref{convE1} and letting $\delta \to 0$ in \eqref{pdelta} we get that
\begin{align}\label{pt}
|P(t)\varphi(x+h)|^p\leq P(t)|\varphi(x)|^p\,{\rm exp}\pa{\frac{pe^{2t\zeta_{\X}}}{t(p-1)}\|h\|_C^2}.
\end{align}
for any $p>1$, $t>0$, $\varphi \in C_b(\X)$, $x \in E$ and $h \in H_C\cap E$. Using the fact that $E$ is densely embedded in $\X$ and the continuity of $P(t)\varphi$ we can extend estimate \eqref{pt} for any $x \in \X$. Finally, using the monotone class theorem as in the proof of Theorem \ref{lip_harnak} we complete the proof.
\end{proof}

\begin{rmk}{\rm
We point out that if $H_C\cap E$ is dense in $H_C$, then \eqref{dae} will hold true for any $h \in H_C$.}
\end{rmk}

\section{Some consequences and examples}\label{sect_examples}

In the first part of this section we collect some corollaries of Theorems \ref{lip_harnak} and \ref{Harnak_diss}. Finally,
we exhibit some classes of examples to which all of our results can be applied.

We start by stating and proving some classical consequences of the (LHI) for which we refer to \cite[Corollary 1.2]{RO-WA1} and \cite[Section 1.3.1]{WANG13}. In the sequel $\mu$ will denote an invariant measure associated with $P(t)$, that is a Borel probability measure on $\X$ such that
\begin{equation*}
\int_{\X} P(t)f  d\mu= \int_{\X} f d\mu, \qquad t>0, \ f \in B_b(\X).
\end{equation*}
Sufficient conditions that guarantee the existence of such a measure can be found in \cite[Chapter 8]{CER1}, \cite[Chapter 6]{DA-ZA3} and \cite[Chapter 11]{DA-ZA3}.
In this case $P(t)$ extends to a positive contractive semigroup in $L^p(\X, \mu)=:L^p_\mu$, the $L^p$ spaces related to the measure $\mu$, (see \cite[Theorem 2.14]{OUH05}).
For simplicity we write $\mu(f)$ to denote $\int_{\X}fd\mu$.
\begin{cor}\label{cor-cons}
Assume that Hypotheses \ref{EU2} and \ref{Rick} hold true.
\begin{enumerate}[\rm (i)]
\item \label{cor_log} For any positive $f\in B_b(\X)$,  $t>0$, $x\in\X$ and $h\in H_C\cap E$
\begin{align}\label{log_version}
[P(t)(\ln f)](x+h)\leq \ln P(t)f(x)+\frac{e^{2t\zeta_{\X}}}{t}\|h\|_C^2.
\end{align}

\item \label{cor_strFeller} For every $f\in B_b(\X)$ and $x\in\X$ it holds
\begin{align}\label{strFeller_property}
\lim_{\substack{\norm{h}_C\ra 0\\ h\in H_C\cap E}}P(t)f(x+h)=P(t)f(x).
\end{align}

\item \label{cor_Wasserstein} The following entropy-cost inequality holds true
\begin{align*}
\mu((P^*(t)f)\ln(P^*(t)f))\leq \frac{e^{2t\zeta_{\X}}}{t}W(f\mu,\mu)^2,
\end{align*}
for any positive function $f\in L^2_\mu$ such that $\mu(f)=1$.  Here $\{P^*(t)\}_{t\geq 0}$ is the adjoint semigroup of $\{P(t)\}_{t\geq 0}$ in $L^2_\mu$ and $W$ denotes the $L^2$-Wasserstein distance with respect to the cost function $(x,y)\mapsto\norm{x-y}_C$, namely for any two probability measure $\mu_1,\mu_2$ on $\X$
\begin{align*}
W(\mu_1,\mu_2)^2:=\inf\set{\int_{\X\times\X}\norm{x-y}_C^2\pi(dx,dy)\tc \pi\in \mathscr{C}(\mu_1,\mu_2)},
\end{align*}
where $\mathscr{C}(\mu_1,\mu_2)$ is the set of all the couplings of $\mu_1$ and $\mu_2$ and we let $\norm{x-y}_C=+\infty$, if $x-y$ does not belong to $H_C\cap E$.
\end{enumerate}
\end{cor}

\begin{proof}
A proof of \eqref{cor_log} can be found in \cite[Section 1.3.1]{WANG13}. We now prove \eqref{cor_strFeller}. It suffices to prove \eqref{strFeller_property} for a non-negative function $f\in B_b(\X)$. Indeed the general case can be obtained writing $f=f^+-f^-$, being $f^{+}$ and $f^-$ the positive and the negative part of $f$. So, let us fix a non-negative function $f$ and for any $\eps>0$ we set $f_\eps:=1+\eps f$. Recalling that $r\leq \ln(1+r)+r^2$ for any $r\geq 0$ we get for every $x\in\X$
\begin{align}\label{YU}
\ln f_\eps(x)=\ln(1+\eps f(x))\geq \eps f(x)-\eps^2f^2(x)\geq \eps f(x)-\eps^2\norm{f}_\infty^2.
\end{align}
Now applying \eqref{log_version} to $f_\eps$, using \eqref{YU} and dividing by $\eps$ we get for every $x\in\X$ and $h\in H_C\cap E$
\begin{align}\label{GI}
P(t)f(x+h)-\eps\norm{f}_\infty^2\leq \frac{1}{\eps}\ln P(t)(1+\eps f(x))+\frac{e^{2t\zeta_{\X}}}{\eps t}\|h\|_C^2.
\end{align}
Taking the supremum limit as $\norm{h}_C\ra 0$ with $h\in H_C\cap E$ and then letting $\eps\to 0$ we get
\begin{align*}
\limsup_{\substack{\norm{h}_C\ra 0\\ h\in H_C\cap E}} P(t)f(x+h)\leq P(t)f(x).
\end{align*}
Recalling that $\ln(1+r)\leq r$ for any $r>-1$ and arguing as above we get that for any $\eps>0$, $x\in\X$ and $h\in H_C\cap E$
\begin{align*}
P(t)\pa{\frac{1+\eps f(x)}{\eps}}-\frac{e^{2t\zeta_{\X}}}{\eps t}\|h\|_C^2\leq \frac{1}{\eps}\ln P(t)(1+\eps f(x-h))\leq P(t)f(x-h).
\end{align*}
Taking the infimum limit as $\norm{h}_C\ra 0$ with $h\in H_C\cap E$ and then letting $\eps\to 0$ we get
\begin{align}\label{OH}
P(t)f(x)\leq \liminf_{\substack{\norm{h}_C\ra 0\\ h\in H_C\cap E}} P(t)f(x-h).
\end{align}
Since $H_C\cap E$ is a linear space then applying \eqref{OH} to $-h$ we get
\begin{align}\label{OH*}
P(t)f(x)\leq \liminf_{\substack{\norm{h}_C\ra 0\\ h\in H_C\cap E}} P(t)f(x+h).
\end{align}
By \eqref{GI} and \eqref{OH*} we get \eqref{strFeller_property}.

Now a standard argument allows us to prove \eqref{cor_Wasserstein} for a bounded Borel and positive function $f$ with $\mu(f)=1$. Writing \eqref{log_version} with $P^*(t)f$ in place of $f$, we get
\begin{align}\label{ARC}
[(P(t)f)(\ln P^*(t)f)](x)\leq \ln (P(t)P^*(t)f(y))+\frac{e^{2t\zeta_{\X}}}{t}\|x-y\|_C^2
\end{align}
for any $t>0$, $x,y\in\X$ such that $x-y\in H_C\cap E$.
Integrating both sides of \eqref{ARC} with respect to $\pi\in\mathscr{C}(f\mu,\mu)$ we get
\begin{align*}
\mu((P^*(t)f)(\ln P^*(t)f))\leq \mu(\ln P(t)P^*(t)f)+\frac{e^{2t\zeta_{\X}}}{t}\int_{\X\times \X}\|x-y\|_C^2\pi(dx,dy).
\end{align*}
To conclude it is sufficient to observe that the Jensen inequality yields that
\begin{gather*}
\mu(\ln P(t)P^*(t)f)\leq \ln\mu(P(t)P^*(t)f)=\ln\mu(f)=0,
\end{gather*}
whence the claim.
\end{proof}
\begin{rmk}{\rm We stress that Corollary \ref{cor-cons} remains true (with the constant $e^{2t\zeta_\X}$ replaced by $K^2_1$) if we assume that the hypotheses of Theorem \ref{lip_harnak} hold true.}
\end{rmk}

Another classical consequence of (LHI) is a hypercontractivity type estimate for the semigroup $P(t)$ in $L^p_\mu$. Such estimate relies on the H\"older inequality and some integrability conditions with respect to $\mu$ of some exponential functions.

\begin{cor}\label{cor-hyp}
Assume that the hypotheses of Theorem \ref{lip_harnak} hold true and that there exists an invariant measure $\mu$ for the semigroup $P(t)$.
If, in addition there exists $\varepsilon>0$ such that
\begin{equation}\label{l_a}
\int_{\X}\int_{\X} e^{\varepsilon \|x-y\|^2_C}\mu(dx)\mu(dy)<+\infty,
\end{equation}
then, for any $p\ge 2$, there exists $t_0>0$ and a positive constant $C$ such that
\begin{equation}\label{hyp_est}
\|P(t)f\|_{L^p_\mu}\le C \|f\|_{L^2_\mu}
\end{equation}
for any $t \ge t_0$ and any $f\in L^2_\mu$.
\end{cor}
\begin{proof}
Let us consider $f \in L^2_\mu$ and $\vartheta \in (1,2)$. By \eqref{lip_harnak_est} we deduce that for any $t>0$
\begin{align*}
\int_{\X}|P(t)f|^{2\vartheta}(x)\mu(dx)=&\int_{\X}\int_{\X}|P(t)f|^{2\vartheta}(x)\mu(dx)\mu(dy)\\
=& \int_{\X}\int_{\X}|P(t)f(x)|^{\vartheta}(|P(t)f(x)|^2)^{\vartheta/2}\mu(dx)\mu(dy)\\
\le & \int_{\X}\int_{\X}|P(t)f(x)|^{\vartheta}(P(t)f^2(y))^{\vartheta/2}e^{\frac{\vartheta K_1^2}{t}\|x-y\|^2_C}\mu(dx)\mu(dy)\\
=:& \int_{\X}\int_{\X}h(x,y)g(x,y)\mu(dx)\mu(dy),
\end{align*}
where $h(x,y):=|P(t)f(x)|^{\vartheta}(P(t)f^2(y))^{\vartheta/2}$ and $g(x,y):=e^{\frac{\vartheta K_1^2}{t}\|x-y\|^2_C}$. Applying the H\"older inequality with respect to the measure $\mu\otimes \mu$ we get
\[\int_{\X}\int_{\X}h(x,y)g(x,y)\mu(dx)\mu(dy)\le \|h\|_{L^{2/\vartheta}_{\mu\otimes \mu}}\|g\|_{L^{2/(2-\vartheta)}_{\mu\otimes \mu}}.\]
Now, the invariance of $\mu$ and the contractivity of $P(t)$ in $L^2_\mu$ allow us to estimate
\begin{align*}
\|h\|_{L^{2/\vartheta}_{\mu\otimes \mu}}= \|P(t)f\|_{L^2_\mu}^\vartheta \|f\|_{L^2_\mu}^\vartheta
\le \|f\|_{L^2_\mu}^{2\vartheta}.
\end{align*}
Moreover, being
\begin{align*}
\|g\|_{L^{2/(2-\vartheta)}_{\mu\otimes \mu}}^{2/(2-\vartheta)}= \int_{\X}\int_{\X}e^{\frac{2\vartheta}{2-\vartheta}\frac{K_1^2}{t}\|x-y\|^2_C}\mu(dx)\mu(dy)=:C(\vartheta,t)
\end{align*}
condition \eqref{l_a} ensures that there exists $\overline{t}>0$ such that $C(\vartheta,t)<+\infty$ for any $t \ge \overline{t}$ and any $\vartheta \in (1,2)$. Consequently
\begin{align*}
\|P(t)f\|_{L^{2\vartheta}_\mu}\le (C(\vartheta,t))^{\frac{2-\vartheta}{4\vartheta}}\|f\|_{L^2_\mu}.
\end{align*}
i.e., $P(t)$ maps $L^2_\mu$ into $L^{2\vartheta}_\mu$ for $t\ge \overline{t}$. Since $\vartheta>1$, $P(t)$ actually improves summability of the initial datum when $t\ge\overline{t}$. To go further we use the semigroup law. Indeed, if $f \in L^2_\mu$, then $P(\overline{t})f \in L^{2\vartheta}_\mu$, i.e. $|P(\overline{t})f|^{\vartheta}\in L^2_\mu$. Using again the first part of the proof, we deduce that
$P(t)|P(\overline{t})f|^{\vartheta}\in L^{2\vartheta}_\mu$ for $t \ge \overline{t}$. Since, by the Jensen inequality and the positivity of $P(t)$ we can estimate
\begin{align*}
+\infty>\|P(t)|P(\overline{t})f|^{\vartheta}\|_{L^{2\vartheta}_\mu}^{2\vartheta}=&\int_{\X}|P(t)|P(\overline{t})f|^{\vartheta}|^{2\vartheta}d\mu\\
\ge & \int_{\X}|P(t)|P(\overline{t})f||^{2\vartheta^2}d\mu\\
= & \int_{\X}(P(t)|P(\overline{t})f|)^{2\vartheta^2}d\mu\\
\ge & \int_{\X}|P(t)P(\overline{t})f|^{2\vartheta^2}d\mu\\
= & \int_{\X}|P(t+\overline{t})f|^{2\vartheta^2}d\mu, \qquad\;\, t \ge \overline{t}
\end{align*}
we infer that $P(t)$ maps $L^2_\mu$ into $L^{2\vartheta^2}_\mu$ for any $t \ge 2\overline{t}$. Iterating this procedure we can prove that for any $p>2$ there exists $t_0=t_0(p)>0$ such that $P(t)$ maps $L^2_\mu$ into $L^p_\mu$ for any $t \ge t_0$ and estimate \eqref{hyp_est} holds true.
\end{proof}

\begin{rmk}{\rm Note that the result in Corollary \ref{cor-hyp} and \eqref{hyp_est} continue to hold true if we assume the hypotheses of Theorem \ref{Harnak_diss}, that $\zeta_{\X}<0$ and that \eqref{l_a} is satisfied with $\varepsilon=1$.}
\end{rmk}

Now we collect some examples of operators to which the results of Sections \ref{liphar} and \ref{nonome} and of Corollaries \ref{cor-cons} and \ref{cor-hyp} can be applied.

\begin{example}{\bf{A suitable choice for $A$ and $C$}.}\\\label{suitable}{\rm
Let $\X$ be a Hilbert space. For a positive and compact operator $Q\in \mathcal{L}(\X)$, we set
\[A:=-(1/2)Q^{-\beta}:Q^{\beta}(\X)\subseteq\X\ra\X,\qquad  C:=Q^{2\alpha},\]
with $\alpha,\beta\geq 0$ such that \eqref{condprato} is verified.

Let $\{e_k\}_{k \in\N}$ be an orthonormal basis of $\X$ consisting of eigenvectors of $Q$, and let $\{\lambda_k\}_{k\in\N}$ be the eigenvalues associated with $\{e_k\}_{k \in\N}$. Since $Q$ is a compact and positive operator, there exists $k_0\in\N$ such that  $0<\lambda_k\leq \lambda_{k_0}$, for any $k\in\N$. Without loss of generality we can assume that $k_0=1$. Hence, for any $x\in Q^{\beta}(\X)$
\begin{equation}\label{disAX}
\scal{Ax}{x}=-\frac{1}{2}\sum_{k=1}^{+\infty}\lambda^{-\beta}_k \scal{x}{e_k}^2\leq -\frac{1}{2}\lambda^{-\beta}_1\norm{x}^2.
\end{equation}
Moreover, by the properties of $Q$, $\Dom(A)=Q^{\beta}(\X)$ is dense in $\X$, so $A$ generates a strongly continuous and contraction semigroup in $\X$ and, so Hypotheses \ref{EU2} are satisfied. Further, let us consider $A_{C}$, the part of $A$ in $H_C=H_{Q^\alpha}$ and recall that
\[
\Dom(A_{C}):=\{x\in Q^{\alpha}(\X)\cap Q^{\beta}(\X)\,|\, Ax\in Q^{\alpha}(\X)\}.
\]
By \eqref{disAX}, for any $x\in \Dom(A_C)$, we have
\begin{equation}\label{disAalfa}
[Ax,x]_C=\langle Q^{-\alpha}Ax, Q^{-\alpha}x\rangle =\langle AQ^{-\alpha}x,Q^{-\alpha}x\rangle \leq -\frac{\|Q^{-\alpha}x\|^2}{2\lambda^{\beta}_1}=-\frac{\norm{x}_C^2}{2\lambda^{\beta}_1}.
\end{equation}
Since $Q^{\alpha+\beta}(\X)$ is dense in $\X$ and $Q^{-\alpha}$ is a close operator in $\X$, then $Q^{\alpha+\beta}(\X)$ is dense in $H_C$, moreover $Q^{\alpha+\beta}(\X)\subseteq \Dom(A_C)$. Hence $A$ generates a strongly continuous and contraction semigroup in $H_C$. By the compactness and positivity of $Q$, we get that $H_{C}$ is dense in $\X$. Moreover Hypotheses \ref{xn} is verified with $\mathcal{X}_n:=\mbox{span}\{e_1,\ldots,e_n\}$.
Let $G$ be a Lipschitz continuous function, it easy to see that $F=Q^{\alpha}G$ is Lipschitz continuous too with Lipschitz constant $\norm{Q^{\alpha}}_{\mathcal{L}(\X)}L_G$. Moreover due to the choice of $Q$ we have $\norm{Q^{\alpha}}L_G\leq \lambda_1^{\alpha}L_G$
and so
\begin{align*}
&\scal{Q^\alpha\J G(x)h}{h}\leq \lambda_1^{\alpha}L_G\norm{h}^2,\qquad x,h\in \X.
\end{align*}
If we assume $L_G<(2\lambda_1^{\alpha+\beta})^{-1}$ then condition \eqref{diss_X} is verified. So Hypotheses \ref{Sinem}(ii) and Hypotheses \ref{xn} are satisfied too. If, in addition Hypotheses \ref{Sinem}(i) are satisfied, then Theorem \ref{lip_harnak} can be applied. This is the case when, for instance, $(A, \Dom(A))$ is a sectorial operator in $\X$ as the next example shows.
}
\end{example}

\begin{example}{\bf{An example in $L^2([0,1],\lambda)$.}}\label{suitSOB}\\{\rm
In the previous example one can take as $\X$ the Hilbert space $L^2([0,1],\lambda)$ where $\lambda$ is the Lebesgue measure and by $-Q^{-1}$ the realization in $L^2([0,1],\lambda)$ of the second order derivative with Dirichlet boundary conditions. With these choices $Q$ turns out to be a positive and trace class operator in $\X$. Thus, if $A$ and $C$ are defined as above, then estimate \eqref{disAX} and \eqref{disAalfa} hold true with $\lambda_1= \pi^{-2}$ and $A$ generates a strongly continuous and analytic semigroup $e^{tA}$ satisfying
\[
\|e^{tA}\|_{\mathcal{L}(\X)}\leq e^{-(1/2)\pi^2t},\qquad t\geq 0
\]
see \cite[Chapter 4]{DA-ZA3}. Moreover by \cite[Proposition 2.1.1]{LUN1}, if $\alpha<\beta$ then \eqref{case1} is verified with $\gamma=\alpha/\beta$.
If, in addition $F$ is a Lipschitz continuous function with Lipschitz constant $L_F<(1/2)\pi^{2\beta}$, then \eqref{diss_X} is verified. In particular Hypotheses \ref{Sinem}(i) and Hypotheses \ref{xn} are verified. Choosing $\alpha=1/2<\beta$, then $H_C$ is the space $W_0^{1,2}([0,1],\lambda)$ and \eqref{condprato} is satisfied too and Theorem \ref{lip_harnak} can be applied.}
\end{example}

\begin{example}{\bf{Infinite dimensional polynomial.}}\\{\rm
In this example we exhibit a class of functions satisfying Hypotheses \ref{Rick}.
To this aim we recall the notion of infinite dimensional polynomial (see \cite{CHAE1,DIN1,MUJ1}). For every $n\in\N$, we say that a map $V:\X^n\ra\X$ is $n$-multilinear if it is linear in each variable separately. A $n$-multilinear map $V$ is said to be symmetric if
\begin{equation}\label{simme}
V(x_1,\ldots,x_n)=V(x_{\sigma(1)},\ldots,x_{\sigma(n)}),
\end{equation}
for any permutation $\sigma$ of the index set $\{1,\ldots,n\}$. We say that a function $P_n:\X\ra\X$ is a homogeneous polynomial of degree $n\in\N$ if there exists a $n$-multilinear symmetric map $V$ such that for every $x\in\X$
$$
P_n(x)=V(x,\ldots,x).
$$
We consider the function $F:\X\ra\X$ defined by
\begin{align*}
F(x):=P_n(x)+\zeta_F x,
\end{align*}
where $x\in\X$, $\zeta_F\in\R$ and $P_n$ is a homogeneous polynomial of degree $n$ such that,
\begin{align}\label{abyss}
\gen{V(h,x,\ldots,x),h}\leq 0,
\end{align}
for any $h, x \in \X$. By \cite[Theorem 3.4]{CHAE1}, there exists $d>0$ such that
\begin{equation}\label{crepol}
\norm{F(x)}\leq d(1+\norm{x}^n),\qquad x\in\X.
\end{equation}
Moreover, for any $x,h\in\X$, we have
\begin{equation*}
\J P_n(x)h=nV(h,x,\ldots,x),
\end{equation*}
and so, by \eqref{abyss}, for any $x,y\in\X$, we obtain
\begin{equation}\label{dispol}
\scal{F(x)-F(y)}{x-y}\leq \zeta_F\norm{x-y}^2.
\end{equation}

In order to give a more concrete application we place ourself in the same scenario as in Example \ref{suitSOB} with $\alpha=1/2<\beta$. Consider a version of $K\in L^2([0,1]^4,\lambda)$ which is symmetric (see \eqref{simme}) and set
\begin{align}\label{polpol}
[P_3(f)](\xi):=\int_0^1\int_0^1\int_0^1 K(\xi_1,\xi_2,\xi_3,\xi) f(\xi_1)f(\xi_2)f(\xi_3)d\xi_1 d\xi_2 d\xi_3
\end{align}
with $f\in L^2([0,1])$. $P_3$ is a homogeneous polynomial of degree three on $L^2([0,1],\lambda)$ (see \cite[Exercise 1.73]{DIN1}). Moreover, condition \eqref{abyss} holds true whenever $K$ has negative values. Indeed observe that, for $f_1,f_2,f_3\in L^2([0,1],\lambda)$,
\[V(f_1,f_2,f_3)=\int_0^1\int_0^1\int_0^1 K(\xi_1,\xi_2,\xi_3,\xi) f_1(\xi_1)f_2(\xi_2)f_3(\xi_3)d\xi_1 d\xi_2 d\xi_3,\]
and for $f,h\in L^2([0,1],\lambda)$
\[\gen{V(h,f,f),h}=\int_0^1\int_0^1\int_0^1\int_0^1 K(\xi_1,\xi_2,\xi_3,\xi) f(\xi_1)f(\xi_2)h(\xi_3)h(\xi)d\xi_1 d\xi_2 d\xi_3 d \xi.\]
A standard argument allows to deduce that $\gen{V(h,f,f),h}=0$ if, and only if, $f=0$ $\lambda$-a.e. or $h=0$ $\lambda$-a.e. So by the continuity of $\gen{V(h,f,f),h}$ with respect to $h$ (for a fixed $f$) and the fact that
\[\gen{V(-h,f,f),-h}=\gen{V(h,f,f),h},\]
the claim follows. In addition, if we assume that $K$ has weak derivative with respect to the fourth variable belonging to $L^2([0,1]^4,\lambda)$, then for any $f\in W^{1,2}([0,1],\lambda)$ it holds that $F(f)=P_3(f)+\zeta_Ff$ belongs to $W^{1,2}([0,1],\lambda)$ (see \eqref{polpol}) and its weak derivative is
\begin{equation}\label{dedeb}
[F(f)]'(\xi)=\int_0^1\int_0^1\int_0^1 \dfrac{\partial K }{\partial \xi}(\xi_1,\xi_2,\xi_3,\xi) f(\xi_1)f(\xi_2)f(\xi_3)d\xi_1 d\xi_2 d\xi_3+\zeta_F f'(\xi).
\end{equation}
If we assume that $(\partial K/\partial \xi)\in L^2([0,1]^4,\lambda)$ has a symmetric version (see \eqref{simme}) with negative value, then by \eqref{dedeb}, estimate \eqref{crepol} and \eqref{dispol} are verified in $W^{1,2}([0,1],\lambda)$. Hence Hypotheses \ref{Rick}(i) is verified and with the choices of $A$ and $C$ in the Example \ref{suitSOB}, all the results in Section \ref{nonome} can be applied.}
\end{example}

\begin{example}{\bf{A reaction-diffusion system.}}\\ {\rm
Assume that $\X=L^2([0,1],\lambda)$ (where $\lambda$ is the Lebesgue measure), $E=C([0,1])$, $A$ is the realization in $L^2([0,1],\lambda)$ of the second order derivative operator with Dirichlet boundary condition and $C=\Id_\X$.
In order to define the function $F$ we consider a decreasing function $\varphi\in C^1(\R)$ such that
\begin{align*}
\abs{\varphi'(\xi)}\leq d_1(1+\abs{\xi}^{m}),\qquad \xi\in\R,
\end{align*}
for some constants $d_1>0$ and $m \in \N$.
Let $\zeta_F>0$. We set
\[ [F(f)](\xi)=
\left\{\begin{array}{ll}
\varphi(f(\xi))-\frac{\zeta_F}{2}f(\xi)^2, & f\in C([0,1]),\ \xi\in [0,1];\\
0, &\mbox{otherwise}.
\end{array}\right.
\]
By \cite[Section 6.1, Lemma 6.1.2 and Lemma 8.2.1]{CER1} and \cite[Example D.7]{DA-ZA3} it follows that Hypotheses \ref{Rick} are verified, in particular \eqref{diss_X} is satisfied.
Finally, taking into account the results in Examples \ref{suitable} and \ref{suitSOB} (when $\alpha=0$ and $\beta=1$), we can conclude that Hypotheses \ref{Sinem}(i) and Hypotheses \ref{xn} are verified too and so Theorem \ref{Harnak_diss} can be applied.}
\end{example}

%
%
%


\begin{thebibliography}{99}


\bibitem{AROSERR}
Aronson, D. G., Serrin, J.,
\newblock{Local behaviour of solutions of quasi-linear parabolic equations},
\newblock{\em Arch. Rat. Mech. Anal.} {\bf 25} (1967), 81--123.

\bibitem{BI1}
Bignamini, D.A.,
\newblock{$L^2$-theory for transition semigroups associated to dissipative systems.},
\newblock{\em Arxiv e-prints}, 2021. \textbf{arXiv}: 2110.05271.

\bibitem{BI-FE1}
Bignamini, D.A., Ferrari, S.,
\newblock{Regularizing Properties of (Non-Gaussian) Transition Semigroups in Hilbert Spaces},
\newblock{\em Potential Anal.} (2021), \textbf{DOI}:10.1007/s11118-021-09931-2.

\bibitem{CER1}
Cerrai, S.,
\newblock{\em Second order {PDE}'s in finite and infinite dimension},
\newblock{Lecture Notes in Mathematics}, vol. 1762, Springer-Verlag, Berlin, 2001.

\bibitem{CHAE1}
Chae, S.B.
\newblock{\em Holomorphy and Calculus in Normed Spaces},
\newblock{Chapman \& Hall/CRC Pure and Applied Mathematics}, vol. 92, Taylor \& Francis Group, New York and Basel, 1985.

\bibitem{DaPROWA09}
Da Prato, G., and R\"{o}ckner, M., Wang, F.-Y.,
\newblock{Singular stochastic equations on Hilbert spaces: Harnack inequalities for their transition semigroups},
\newblock{\em J. Funct. Anal.} {\bf 257} (2009), 992--1017.

\bibitem{DA-ZA2}
Da Prato, G., Zabczyk, J.,
\newblock{\small \it Ergodicity for Infinite-Dimensional Systems},
\newblock{London Mathematical Society Lecture Note Series}, vol. 229, Cambridge University Press, Cambridge, 1996.

%

\bibitem{DA-ZA3}
Da Prato, G., Zabczyk, J.,
\newblock{\em Stochastic equations in infinite dimensions},
\newblock{Encyclopedia of Mathematics and its Applications}, vol. 152, Second edition, Cambridge University Press, Cambridge, 2014.

\bibitem{DENG18}
Deng, C.S., Zhang, S.Q.,
\newblock{Log-Harnack inequalities for Markov semigroups generated by non-local Gruschin type operators},
\newblock{\em Math. Nachr.} {\bf 291} (2018), 1055--1074.

\bibitem{DEGIO57}
De Giorgi, E.,
\newblock{Sulla differenziabilit\`a e l'analiticit\`a delle estremali degli integrali multipli regolari},
\newblock{\em Mem. Accad. Sci. Torino. Cl. Sci. Fis. Mat. Nat.} {\bf 3} (1957), 25--43.


\bibitem{DiB1}
DiBenedetto, E., Gianazza, U., Vespri, V.,
\newblock{Forward, backward and elliptic harnack inequalities for non-negative solutions to certain singular parabolic partial
differential equations},
\newblock{\em Ann. Sc. Norm. Super. Pisa Cl. Sci. (5)} {\bf 9} (2010) 385--422.

\bibitem{DiB2}
DiBenedetto, E., Gianazza, U., Vespri, V.,
\newblock{Harnack estimates for quasi-linear
degenerate parabolic differential equations},
\newblock{\em Acta Math.} {\bf 200} (2008), 181--209.

\bibitem{DIN1}
Dineen, S.,
\newblock{\em Complex Analysis on Infinite Dimensional Spaces},
\newblock{Springer Monographs in Mathematics}, Springer-Verlag, London, 1999.




\bibitem{DUN-SCH1}
Dunford, N., Schwartz, J.T.,
\newblock{\em Linear operators. Part I},
\newblock{Wiley Classics Library, John Wiley \& Sons, Inc., New York}, 1988.

\bibitem{DUN-SCH2}
Dunford, N., Schwartz, J.T.,
\newblock{\em Linear operators. Part II},
\newblock{Wiley Classics Library, John Wiley \& Sons, Inc., New York}, 1988.

\bibitem{EN-NA1}
Engel, K.-J., Nagel, R.,
\newblock{\em A short course on operator semigroups},
\newblock{Universitext}, Springer, New York, 2006.

\bibitem{ES09}
Es-Sarhir, A., von Renesse, M.-K., Scheutzow, M.,
\newblock{Harnack inequality for functional {SDE}s with bounded memory},
\newblock{\em Electron. Commun. Probab.} {\bf 14} (2009), 560--565.

\bibitem{GOKO01}
Goldys, B., Kocan, M.,
\newblock{Diffusion semigroups in spaces of continuous functions with mixed topology},
\newblock{\em J. Differential Equations} {\bf 173} (2001), 17--39.

\bibitem{Gordina11}
Gordina, M., R\"{o}ckner, M., Wang, F.-Y.,
\newblock{Dimension-independent {H}arnack inequalities for subordinated semigroups},
\newblock{\em Potential Anal.} {\bf 34} (2011), 293--307.


\bibitem{HAD}
Hadamard, J.,
\newblock{Extension \`a l'equation de la chaleur d'un th\`eor\`eme de A. Harnack},
\newblock{\em Rend. Circ. Mat. Palermo} {\bf 2} (1954), 337--346.

\bibitem{HAR87}
Harnack, A.,
\newblock{\em Die Grundlagen der Theorie des logarithmischen Potentiales und der eindeutigen Potentialfunktion in der Ebene},
\newblock{V. G. Teubner}, Leipzig, 1887.

\bibitem{HUZA19}
Huang, X., Zhang, S.Q.,
\newblock{Mild solutions and {H}arnack inequality for functional stochastic partial differential equations with Dini drift},
\newblock{\em J. Theoret. Probab.} {\bf 32} (2019), 303--329.

\bibitem{JACPRO03}
Jacod, J., Protter, P.,
\newblock{\em Probability essentials},
\newblock{Universitext}, Springer-Verlag, Berlin, 2003.

\bibitem{KAS07}
Kassmann, M.,
\newblock{\em Harnack inequalities: an introduction},
\newblock{Bound. Value Probl.} (2007), Art. ID 81415, 21.


\bibitem{lorbook}
Lorenzi, L.,
\newblock{\em Analytical Methods for Markov Semigroups}, CRC Press, second ed., 2017.

\bibitem{LUN1}
Lunardi, A.
\newblock{\em Analytic semigroups and optimal regularity in parabolic problems},
\newblock{Modern Birkh\"{a}user Classics, Birkh\"{a}user/Springer Basel AG}, Basel, 1995.

\bibitem{LVHU21}
Lv, W., Huang, X.,
\newblock{Harnack and shift {H}arnack inequalities for degenerate (functional) stochastic partial differential equations with singular drifts},
\newblock{\em J. Theoret. Probab.} {\bf 34} (2021), 827--851.

\bibitem{MAS08}
Masiero, F.,
\newblock{Stochastic optimal control problems and parabolic equations in {B}anach spaces},
\newblock{\em SIAM J. Control Optim.} {\bf 47} (2008), 251--300.

\bibitem{MOS61}
Moser, J.,
\newblock{On {H}arnack's theorem for elliptic differential equations},
\newblock{\em Comm. Pure Appl. Math.} {\bf 14} (1961), 577--591.


\bibitem{MOS67}
Moser, J.,
\newblock{A Harnack inequality for parabolic differential equations},
\newblock{\em Comm. Pure
Appl. Math.} {\bf 17} (1964), 101--134.


\bibitem{MUJ1}
Mujica, J.,
\newblock{\em Complex Analysis in Banach Spaces},
\newblock{North-Holland Mathematics Studies, vol. 120, North-Holland}, Amsterdam, 1985.

\bibitem{NASH57}
Nash, J.,
\newblock{Parabolic equations},
\newblock{\em Proc. Nat. Acad. Sci. U.S.A.} {\bf 43} (1957), 754--758.

\bibitem{OUH05}
Ouhabaz, E.M.,
\newblock{\em Analysis of heat equations on domains},
\newblock{London Mathematical Society Monographs Series, vol. 31, Princeton University Press}, Princeton, NJ, 2005.

\bibitem{PHE1}
Phelps, R.R.,
\newblock{\em Convex Functions, Monotone Operators and Differentiability},
\newblock{Lecture Notes in Mathematics, vol. 1364, Springer-Verlag}, Berlin Heidelberg, 1993.


\bibitem{PIN}
Pini, B.,
\newblock{Sulla soluzione generalizzata di Wiener per il primo problema di valori al
contorno nel caso parabolico},
\newblock{\em Rend. Sem. Mat. Univ. Padova} {\bf 23} (1954), 422--434.

\bibitem{RE-SI1}
Reed, M., Simon, B.,
\newblock{\em Methods of modern mathematical physics. I. Functional analysis},
\newblock{Academic Press}, New York-London, 1972.


\bibitem{RO-WA03}
R\"{o}ckner, M., Wang, F.-Y.,
\newblock{Harnack and functional inequalities for generalized {M}ehler semigroups},
\newblock{\em J. Funct. Anal.} {\bf 203} (2003), 237--261.

\bibitem{RO-WA1}
R\"{o}ckner, M., Wang, F.-Y.,
\newblock{Log-{H}arnack inequality for stochastic differential equations in {H}ilbert spaces and its consequences},
\newblock{\em Infin. Dimens. Anal. Quantum Probab. Relat. Top.} {\bf 13} (2010), 27--37.

\bibitem{SER55}
Serrin, J.,
\newblock{On the {H}arnack inequality for linear elliptic equations},
\newblock{\em J. Analyse Math.} {\bf 4} (1955/56), 292--308.


\bibitem{TRU}
Trudinger, N.S.,
\newblock{On Harnack type inequalities and their application to quasilinear
elliptic equations},
\newblock{\em Commun. Pure Appl. Math.} {\bf 20} (1967), 721--747.

\bibitem{WANG97}
Wang, F.-Y.,
\newblock{Logarithmic {S}obolev inequalities on noncompact {R}iemannian manifolds},
\newblock{\em Probab. Theory Related Fields} {\bf 109} (1997), 417--424.

\bibitem{WANG04}
Wang, F.-Y.,
\newblock{Equivalence of dimension-free {H}arnack inequality and curvature condition},
\newblock{\em Integral Equations Operator Theory} {\bf 48} (2004), 547--552.

\bibitem{WANG13}
Wang, F.-Y.,
\newblock{\em Harnack inequalities for stochastic partial differential equations},
\newblock{SpringerBriefs in Mathematics}, Springer, New York, 2013.

\bibitem{WANGYUAN11}
Wang, F.-Y., Yuan, C.,
\newblock{Harnack inequalities for functional {SDE}s with multiplicative noise and applications},
\newblock{\em Stochastic Process. Appl.} {\bf 121} (2011), 2692--2710.

\end{thebibliography}
\end{document}